\documentclass[pamm,a4paper,fleqn]{w-art}
\usepackage{times,cite,w-thm}
\usepackage[T1]{fontenc}
\usepackage[utf8]{inputenc}
\usepackage{mathtools}
\usepackage{amsmath,amssymb,amsthm}
\usepackage{bm}
\usepackage{enumitem}
\usepackage{xcolor}
\usepackage{multicol}

\def\softd{{\leavevmode\setbox1=\hbox{d}%
		\hbox to 1.05\wd1{d\kern-0.4ex{\char039}\hss}}}

\newlist{assump}{enumerate}{1}
\setlist[assump,1]{label=(A\arabic*),ref=(A\arabic*),leftmargin=*, labelsep=1em}
\setlist[assump,2]{label=(A\arabic{assumpi}.\arabic*),ref=(A\arabic{assumpi}.\arabic*)}
\setlist[assump]{resume}
\newlist{bcs}{enumerate}{1}
\setlist[bcs,1]{label=(BC\arabic*),ref=(BC\arabic*)}
\theoremstyle{plain}

\theoremstyle{definition}

\usepackage{graphicx}

\def\RR{\mathbb{R}}

\def\M{M}
\def\K{K}

\def\I{\mathbf{I}}

\def\u{\mathbf{u}}
\def\vv{\mathbf{v}}

\def\vtheta{\vartheta}
\def\Th{\mathcal{T}_h}
\def\Vh{\mathcal{V}_h}
\def\Qh{\mathcal{Q}_h}
\def\Xh{\mathcal{X}_h}
\def\Itau{\mathcal{I}_\tau}

\def\dt{\partial_t}

\def\Du{\mathrm{D}\u}
\def\Dv{\mathrm{D}\vv}
\def\div{\operatorname{div}}

\def\la{\langle}
\def\ra{\rangle}

\DeclarePairedDelimiter{\norm}{\|}{\|}
\DeclarePairedDelimiter{\snorm}{|}{|}

\begin{document}
%% \def\leftmark{Session title}
%%
%%    The information for the title page will be placed between
%%    \begin{document} and \maketitle. The order of most entries
%%    is determined by the class file and can not be changed by
%%    rearranging them. The maketitle command follows after the
%%    abstract.
%%
%%    Most of the following commands will be completed by the publisher.
%%
%%    \renewcommand{\copyrightyear}{2016}
%%    \DOIsuffix{pamm.20161zzzz}
%%    \Volume{16} 
%%    \Year{2016} 
%%    \pagespan{1}{}
%%
%%    The short title is optional:

\TitleLanguage[EN]
\title{Entropy-stable and energy-conservative fully-discrete finite element method for non-isothermal phase-field models}

%% Please do not enter footnotes or \inst{}-notes into the optional
%% argument of the author command. 

%% Please delete not needed author entries.
%% Information for the first author.
\author{\firstname{Aaron}  \lastname{Brunk}\inst{1,}%
  \footnote{Corresponding author: \ElectronicMail{abrunk@uni-mainz.de},
            phone +49 6131-39-24355}}

\address[\inst{1}]{\CountryCode[DE] Johannes Gutenberg University Mainz}
%%
%%    Information for the second author
\author{\firstname{Dennis} \lastname{Höhn}\inst{1}}%
  %\footnote{Second author footnote.}}

\
%%    Abstract is required.
\AbstractLanguage[EN]
\begin{abstract}
This work presents a conforming finite-element scheme for non-isothermal phase-field systems coupled to the incompressible Navier-Stokes equations. The proposed numerical scheme preserves entropy production and total energy conservation exactly by variable transformations using entropy as main variable instead of temperature. Convergence tests in space are conducted, and representative examples are provided to
demonstrate the scheme’s effectiveness.  
\end{abstract}
%% maketitle must follow the abstract.
\maketitle                   % Produces the title.

\section{Introduction}
    Fundamental processes in engineering involve melting/solidification or thermally driven flow of fluid or gas bubbles. These processes encompass various physical mechanisms and a complex interplay of effects, such as undercooling, surface tension, and latent heat. Conventionally, these phenomena are modelled using a sharp-interface representation~\cite {Rubinstein1971,Alexiades1992}. However, these descriptions often lack the possibility to allow topological changes, droplet coalescence or breakup, which might play a significant role.
    Usually, these limitations are overcome by considering phase-field models, cf.  \cite{ProvatasElder2010,Steinbach2009} and cf. \cite{KarmaRappel1998,BeckermannKarma1997,Liang}, often in combination with an equation for fluid flow. In the following, we look at an example of such models by considering a generalised phase-field equation $\phi$, an incompressible flow $\u$ and a heat transport $\vtheta$ coupled by the system
    \begin{align}
        \dt \phi &+ \div(\phi\u) - \div(M\nabla\tfrac{\mu}{\vtheta}) + N\tfrac{\mu}{\vtheta} = 0,  & \tfrac{\mu}{\vtheta} &=  \tfrac{\partial_\phi f(\phi,\vtheta)}{\vtheta} - \div(\tfrac{\kappa(\vtheta)}{\vtheta}\nabla\phi), \label{eq:sys1}
        \\
        \dt\u &+ (\u\cdot\nabla)\u - \div(\eta\Du - p\I - \bm{\sigma}) = 0,  & 0 &= \div(\u), \label{eq:sys2}\\
        \dt e(\phi,\vtheta) &+ \div(e(\phi,\vtheta)\u) + \div(K\nabla\tfrac{1}{\vtheta}) - (\eta\Du - p\I - \bm{\sigma}):\nabla\u = 0, & \bm{\sigma} &= \kappa(\vtheta)\nabla\phi\otimes\nabla\phi . \label{eq:sys3}
    \end{align}
    which is determined via the Helmholtz free energy given by
    \begin{equation*}
        F(\phi,\nabla\phi,\vtheta) = f(\phi,\vtheta) + \tfrac{\kappa(\vtheta)}{2}\snorm{\nabla\phi}^2.
    \end{equation*}
    From this, we can derive the functional relation for entropy, internal energy, total energy, (weighted) chemical potential and Korteweg stress as
    \begin{align}
        s(\phi,\nabla\phi,\vtheta) &= \partial_\vtheta F(\phi,\nabla\phi,\vtheta) = - \partial_\vtheta f(\phi,\vtheta) - \frac{\kappa'(\vtheta)}{2}\snorm{\nabla\phi}^2, \\
        e(\phi,\nabla\phi,\vtheta) &= F+\vtheta s = f(\phi,\vtheta)- \vtheta\partial_\vtheta f(\phi,\vtheta) + \frac{\kappa(\vtheta)-\vtheta\kappa'(\vtheta)}{2}\snorm{\nabla\phi}^2,\\
        e_{\mathrm{tot}}(\phi,\nabla\phi,\vtheta) &= e(\phi,\nabla\phi,\vtheta) + \frac{1}{2}\snorm{\u}^2, \qquad 
        \tfrac{\mu}{\vtheta} = \tfrac{\delta F/\vtheta}{\delta\phi}, \qquad \bm{\sigma} = \partial_{\nabla\phi}F\otimes \nabla\phi.
    \end{align}
    
    The system must be complemented by suitable initial and boundary data and can be shown to be entropy-productive and total-energy-conservative, given suitable boundary conditions.
    
    To use and rely on these models' predictions, stable, robust numerical methods are necessary. In energetic formulation, i.e. the internal energy equation is used together with the inverse temperature as state variable, we refer to \cite{Brunk2025_noniso,Brunk25Enu}. In the entropic formulation, i.e., when the entropy equation is used together with the temperature as a state variable, similar schemes have been developed \cite{Hoehn25,Hoehn26}.
    All the above schemes preserve entropy production, whereas energy or total energy conservation was shown only for the energetic formulations. Notably, in the energetic formulation, entropy production contains an additional numerical diffusion term, while the entropic formulation is exact. A more in-depth literature review can be found in \cite{Hoehn26}

    The aim of this work is to develop a numerical method that preserves exact entropy production and exact total energy conservation. The main idea is to replace the temperature as the main variable with the entropy, i.e., to reverse the usual Legendre transformations.

    The structure of the paper is as follows. In Section 2, we recall equivalent formulations of the internal energy equation and show entropy production and total energy conservation for the system under consideration. In Section 3,, we deduce a suitable variational formulation via a change of variables, which enables subsequent numerical approximation and the preservation of its structure in Section 4. Section 5 illustrates the method by convergence tests and application-inspired numerical tests.

\section{Equivalent formulations and inherent structures}

    The equation for heat transfer \eqref{eq:sys3} can be reformulated in various ways as an internal energy equation, a total energy equation, or an entropy equation. We first start to deduce the equation for total energy and compute
    \begin{align}
     \frac{1}{2}\left(\dt \snorm{\u}^2 + \div(\snorm{\u}^2\u)\right) = (\dt\u + (\u\cdot\nabla)\u\cdot\u ) = \div(\eta\Du - p\I- \bm{\sigma})\cdot\u.  \label{eq:usqaured}
    \end{align}
    
    Combination of \eqref{eq:sys3} with \eqref{eq:usqaured} yields the total energy equation given by
    \begin{align}
     \dt e_{\mathrm{tot}} + \div(e_{\mathrm{tot}}\u) - \div\left(\eta\Du\cdot\u - p\u - \bm{\sigma}\cdot\u \right) = 0.   \label{eq:totalenergy}
    \end{align}
    
    Next,, we will deduce the equation for entropy and recall well-known thermodynamic identities, such as $s=\tfrac{e-f}{\vtheta}$. With this, we compute the temporal derivative of entropy as follows
    \begin{align}
     \dt s = \tfrac{1}{\vtheta}(\dt e - \dt F) - \tfrac{s}{\vtheta}\dt\vtheta = \tfrac{1}{\vtheta}(\dt e - \dt \phi \partial_\phi f - \dt\nabla\phi \kappa(\vtheta)\nabla\phi) = \dt e \cdot\tfrac{1}{\vtheta} - \dt\phi \tfrac{\mu}{\vtheta} - \div(\dt\phi \tfrac{\kappa(\vtheta)}{\vtheta}\nabla\phi).  \label{eq:entropytd}
    \end{align}
    
    For the entropy transport, we follow the same computations as in  \cite{Hoehn26}, which yields 
    \begin{align}
     \div(s\u) = \u\cdot \nabla\left(\tfrac{e-F}{\vtheta}\right) = \u\cdot(\nabla e\cdot\tfrac{1}{\vtheta} - \nabla\phi\tfrac{\mu}{\vtheta}) - \div(\tfrac{1}{\vtheta}\bm{\sigma}\cdot\u) +\tfrac{1}{\vtheta}\bm{\sigma}:\nabla\u \label{eq:entropytrans}
    \end{align}
    
    Combination of \eqref{eq:entropytd} and \eqref{eq:entropytrans} yields the entropy equation given by
    \begin{align}
     \dt s + \div(s\u) =  \left(\dt e +\div(e\u)\right)\tfrac{1}{\vtheta} - \left(\dt\phi +\div(\phi\u)\right)\tfrac{\mu}{\vtheta} - \div(\dt\phi \tfrac{\kappa(\vtheta)}{\vtheta}\nabla\phi+\tfrac{1}{\vtheta}\bm{\sigma}\cdot\u)   +\tfrac{1}{\vtheta}\bm{\sigma}:\nabla\u.
    \end{align}
    Insertion of the partial differential equations \eqref{eq:sys1}--\eqref{eq:sys3} and rearrangement yields
    \begin{align}
     \dt s + \div(s\u) =  -\div(K\nabla\tfrac{1}{\vtheta})\tfrac{1}{\vtheta} - \left(\div(M\nabla\tfrac{\mu}{\vtheta}) - N\tfrac{\mu}{\vtheta}\right)\tfrac{\mu}{\vtheta} + \tfrac{\eta}{\vtheta}\snorm{\Du}^2- \div(\dt\phi \tfrac{\kappa(\vtheta)}{\vtheta}\nabla\phi + \tfrac{1}{\vtheta}\bm{\sigma}\cdot\u). 
    \end{align}
    
    To reveal the exact entropy production, we use the product rule and obtain
    \begin{align}
     \dt s &+ \div(s\u) = \mathcal{P} + \div(\mathbf{J}_S), \qquad  \mathcal{P} =  M\snorm{\nabla\tfrac{\mu}{\vtheta}}^2 + N\snorm{\tfrac{\mu}{\vtheta}}^2 +K\snorm{\nabla\tfrac{1}{\vtheta}}^2 + \tfrac{\eta}{\vtheta}\snorm{\Du}^2, \label{eq:entropyfull}
     \\
     \mathbf{J}_S &= \div\left(-K\nabla\tfrac{1}{\vtheta}\tfrac{1}{\vtheta} - M\nabla\tfrac{\mu}{\vtheta}\tfrac{\mu}{\vtheta}- \dt\phi \tfrac{\kappa(\vtheta)}{\vtheta}\nabla\phi - \tfrac{1}{\vtheta}\bm{\sigma}\cdot\u\right) \label{eq:entropyflux}
    \end{align}
    
    Under suitable boundary conditions we can integrate \eqref{eq:sys1}, \eqref{eq:totalenergy} and \eqref{eq:entropyfull} over the domain $\Omega$ and use integration by parts to obtain:
    \begin{align*}
     \frac{d}{dt}\int_\Omega \phi = 0 \text{ if } N=0, \qquad \frac{d}{dt}\int_\Omega e_{\mathrm{tot}} = 0, \qquad   \int_\Omega s = \int_\Omega \mathcal{P} \geq 0.
    \end{align*}
    
    Note that positive entropy production requires that $K,M,N,\eta$ are positive semi-definite matrices or non-negative functions.
    
    % Comparing to the results in \cite{Marita} we see that the model is consistent with \dennis{GENERIC?}. The entropy equation is a generalisation of the formulation obtained \cite{Dennis1,Dennis2}.
    
    \section{Variational reformulation:}
     
    To facilitate our analysis, we introduce the following set of assumptions:
    
    \begin{assump}
        \item We consider $\Omega \subset \mathbb{R}^d$, $d=1,2,3$ with periodic boundary conditions, i.e. 
            \begin{equation*}
                g(x+L_ie_i) = g(x) \text{ for } g\in\{\phi,\mu,\vtheta,\u,p\} \text{ and } i=1,\ldots, d,
            \end{equation*}
            where $\Omega=[0,L_1]\times\ldots\times[0,L_d]$ is a (hyper)-cube identified with the $d$-dimensional torus $\mathcal{T}^d$.
            Moreover, other functions on $\Omega$ are also assumed to be periodic.
        \item The diffusive mobility $M:=M(\phi,\vtheta)\in\RR$ and mass transfer rate $N:=N(\phi,\vtheta)\in\RR$ are (non)-negative.
        \item The heat conductivity $\K:=\K(\phi,\vtheta)\in\RR $ and the viscosity $\eta:=\eta(\phi,\vtheta)\in\RR$ are non-negative.
        \item The driving potential $f(\cdot,\cdot):\RR\times\RR_+\to \RR$ is smooth and strictly concave in $\vtheta$ for every fixed $\phi$ and goes to infinity for $\vtheta\to 0$. The gradient contribution $\kappa(\vtheta):\RR_+\mapsto\RR_+$ is concave.\label{as:pot}
        \item We introduce the abbreviation $\la a,b \ra=\int_\Omega ab$ and the skew-symmetric form given by $$c_{\mathrm{skw}}(\u,\u,\vv):=\tfrac{1}{2}\la (\u\cdot\nabla)\u,\vv \ra - \tfrac{1}{2}\la (\u\cdot\nabla)\vv,\u \ra$$.
    \end{assump}

    The goal of this section is to derive a variational formulation that allows subsequent discretisation. For this we introduce $\mu:=\vtheta\frac{\mu}{\vtheta}$ as
    \begin{equation}
        \mu  = \partial_\phi f(\phi,\vtheta) - \vtheta\div(\tfrac{\kappa(\vtheta)}{\vtheta}\nabla\phi). \label{eq:defmu}
    \end{equation}
    Next we expand the Korteweg stress $\bm{\sigma}$ using standard thermodynamic identities as follows:
    \begin{align}
     \div(\bm{\sigma}) &= \nabla\phi\otimes\nabla\phi \cdot\kappa'(\vtheta)\nabla\vtheta + \kappa(\vtheta)\Delta\phi\nabla\phi + \tfrac{\kappa(\vtheta)}{2}\nabla\snorm{\nabla\phi}^2 \notag\\
     &=\nabla\phi\otimes\nabla\phi \cdot\kappa'(\vtheta)\nabla\vtheta -\mu\nabla\phi + \partial_\phi f(\phi,\vtheta) - \vtheta(\nabla\phi\cdot\nabla\tfrac{\kappa(\vtheta)}{\vtheta})\nabla\phi + \tfrac{\kappa(\vtheta)}{2}\nabla\snorm{\nabla\phi}^2 
     \notag\\
     & = \tfrac{1}{\vtheta}\bm{\sigma}\cdot\nabla\vtheta -\mu\nabla\phi + \partial_\phi f(\phi,\vtheta) + \tfrac{\kappa(\vtheta)}{2}\nabla\snorm{\nabla\phi}^2 
     \notag\\
     & = \tfrac{1}{\vtheta}\bm{\sigma}\cdot\nabla\vtheta +\phi\nabla\mu - \partial_\vtheta f(\phi,\vtheta)\nabla\vtheta - \tfrac{\kappa'(\vtheta)}{2}\snorm{\nabla\phi}^2\nabla\vtheta + \nabla\left[f(\phi,\vtheta ) + \tfrac{\kappa(\vtheta)}{2}\snorm{\nabla\phi}^2 - \phi\mu\right]  \notag\\
     & = \tfrac{1}{\vtheta}\bm{\sigma}\cdot\nabla\vtheta +\phi\nabla\mu + s\nabla\vtheta + \nabla\left[F(\phi,\nabla\phi,\vtheta )  - \phi\mu\right]. \label{eq:korteweg}
    \end{align}
    Note that the gradient term $\nabla\left[F(\phi,\nabla\phi,\vtheta )  - \phi\mu\right]$ is absorbed into the pressure.
    \begin{align}
        \mu(\phi,\nabla\phi,\vtheta)  &= \delta_\phi F - \vtheta\partial_{\nabla\phi} F\cdot\nabla\tfrac{1}{\vtheta}
        \Longleftrightarrow \mu(\phi,\nabla\phi,s) = \delta_\phi e - \vtheta\partial_{\nabla\phi} e\cdot\nabla\tfrac{1}{\vtheta} = \delta_\phi e + \tfrac{1}{\vtheta}\partial_{\nabla\phi}e\cdot\nabla \vtheta.\label{eq:mutrafo}
    \end{align}
    This allows us to rewrite the system with the state variables $(\phi,\mu,s,\vtheta,\u,p)$ given by the following Lemma.
    \begin{lemma}\label{lem:var}
      Smooth solutions of system \eqref{eq:sys1}--\eqref{eq:sys3}, i.e. $(\phi,\mu,\vtheta,\u,p)$ fulfill after change of variables to $(\phi,\mu,s,\vtheta,\u,p)$ the following variational formulation for smooth test functions $\psi,\xi,\omega,\chi,\vv,q$: 
     \begin{align*}
       \la \dt\phi,\psi \ra &- \la \phi\u,\nabla\psi \ra + \la M\nabla\tfrac{\mu}{\vtheta},\nabla\psi \ra + \la N\tfrac{\mu}{\vtheta},\psi\ra = 0, 
       \\
       \la \mu,\xi \ra &- \la \partial_\phi e(\phi,s),\xi \ra - \la \partial_{\nabla\phi} e,\nabla\xi \ra  - \la \tfrac{1}{\vtheta}\partial_{\nabla\phi} e,\xi\nabla\vtheta \ra =0, 
       \\
       \la \dt s,\omega \ra &- \la s\u,\nabla\omega \ra - \la K\nabla\tfrac{1}{\vtheta},\nabla\tfrac{\omega}{\vtheta} \ra - \la M\nabla\tfrac{\mu}{\vtheta},\nabla\tfrac{\mu\omega}{\vtheta} \ra  - \la N\tfrac{\mu}{\vtheta},\tfrac{\mu\omega}{\vtheta} \ra - \la \tfrac{\eta\omega}{\vtheta} \snorm{\Du}^2\ra \\
       &- \la \dt\phi\partial_{\nabla\phi} e + \bm{\sigma}\cdot\u,\tfrac{1}{\vtheta}\nabla\omega \ra = 0,
       \\
       \la \vtheta,\chi \ra &- \la \partial_s e(\phi,s),\chi \ra = 0, 
       \\
       \la \dt\u,\vv \ra &+ c_{\mathrm{skw}}(\u,\u,\vv) + \la \eta\Du,\Dv \ra - \la p,\div(\vv) \ra + \la (\tfrac{1}{\vtheta}\bm{\sigma} + s\I)\nabla\vtheta,\vv \ra + \la \phi\nabla\mu,\vv\ra = 0,\\
       \la\div(\u),q\ra &= 0.
    \end{align*}
    
    Furthermore, conservation of mass, conservation of total energy and entropy production hold, i.e.
    \begin{align*}
     \frac{d}{dt}\la \phi,1\ra = 0 \text{ if } N=0, \qquad \frac{d}{dt}\la  e_{\mathrm{tot}},1\ra = 0, \qquad   \frac{d}{dt}\la s,1\ra = \la  \mathcal{P},1\ra \geq 0.
    \end{align*}
     The latter properties follow by using only the test functions $\xi=\dt\phi,\chi=\dt s,\psi\in\{1,\mu\},\omega\in\{1,\vtheta\},\vv=\u,q=p.$ 
    \end{lemma}
    
    \begin{proof}
     We replace \eqref{eq:sys3} by \eqref{eq:entropyfull} and introduce \eqref{eq:mutrafo} and $\vtheta=\partial_s e$ as new variables. Standard integration by parts, similarly to \cite{Hoehn25,Hoehn26} and usage of \eqref{eq:korteweg}, yields the variational formulation. Conservation of mass follows by $\psi=1$ and entropy production by $\omega=1$. For total energy conservation we recall that
     \begin{align*}
      \frac{d}{dt}\la e(\phi,\nabla\phi,s),1 \ra = \la \dt s, \partial_s e \ra + \la \dt\phi, \partial_\phi e \ra + \la \dt\nabla\phi,\partial_{\nabla\phi} e \ra  =  \la \dt s, \vtheta \ra + \la \dt\phi, \mu \ra - \la \tfrac{1}{\vtheta}\partial_{\nabla e}\nabla\vtheta,\dt\phi \ra
     \end{align*}
    Here we used the test functions $\chi=\dt s, \xi=\dt\phi$. Next we insert $\omega=\vtheta$ and $\psi=\mu$ to find
     \begin{align}
         \frac{d}{dt}\la e(\phi,\nabla\phi,s),1 \ra &= \la s\u,\nabla\vtheta \ra + \la M\nabla\tfrac{\mu}{\vtheta},\nabla\mu\ra  + \la N\tfrac{\mu}{\vtheta},\mu \ra + \la \eta \snorm{\Du}^2,1\ra \notag\\
       &+ \la\bm{\sigma}\cdot\u,\tfrac{1}{\vtheta}\nabla\vtheta \ra + \la \phi\u,\nabla\mu \ra - \la M\nabla\tfrac{\mu}{\vtheta},\nabla\mu \ra - \la N\tfrac{\mu}{\vtheta},\mu\ra \notag\\
       &=\la s\u,\nabla\vtheta \ra  + \la\bm{\sigma}\cdot\u,\tfrac{1}{\vtheta}\nabla\vtheta \ra + \la \phi\u,\nabla\mu \ra + \la \eta \snorm{\Du}^2,1\ra. \label{eq:weakent}
     \end{align}
     Next, we consider the temporal evolution of the kinetic energy using $\vv=\u$, $q=p$ and skew-symmetry yields
     \begin{align}
       \frac{d}{dt}\la\tfrac{1}{2}\snorm{\u}^2,1\ra = - \la \eta\Du,\Du \ra - \la (\tfrac{1}{\vtheta}\bm{\sigma} + s\I)\nabla\vtheta,\u \ra - \la \phi\nabla\mu,\u\ra   \label{eq:weakkin}
     \end{align}
     A combination of \eqref{eq:weakent} and \eqref{eq:weakkin} yields total energy conservation.
    \end{proof}
    
\section{Numerical Approximation:}

    Based on the variational formulation, we will now deduce a suitable fully-discrete approximation. The discretisation is based on the lowest-order Petrov-Galerkin approximation in time, cf. \cite{Egger,AndrewsA,AndrewsB} for an abstract framework. Such discretisations, although without the necessity of reverting Legendre transforms, are already developed for the Cahn-Hilliard equation, the Cahn-Hilliard-Navier-Stokes system, and a viscoelastic Cahn-Hilliard system, see \cite{BrunkCH,BrunkCHNS,BrunkCHQ}. We note that in these works, it is shown that the method is second-order in time. After explicit evaluation of most integrals in time, the Petrov-Galerkin scheme can be rewritten as a non-linear time-stepping scheme which shares many similarities with Crank-Nicolson or average vector-field approximations. To this end, we introduce the space-time approximation as follows.
    
    \noindent\textbf{Space approximation:}
      A conforming partition $\Th$ of the domain $\Omega$ into simplices of maximal diameter $h$ is considered and $\Th$ can be extended periodically to periodic extensions of $\Omega$. We introduce the standard finite element spaces via
            \begin{align*}
                \Vh &:= \{v \in H^1(\Omega)\cap C^0(\bar\Omega) : v|_K \in P_1(K),~\forall K \in \Th\},\\
                \Xh &:= \{v \in H^1(\Omega)^d\cap C^0(\bar\Omega)^d : v|_K \in P_2(K)^d,~ \forall K \in \Th\}, \\
                \Qh &:= \{v \in \Vh : \la v, 1\ra=0\}.
            \end{align*}
    \textbf{Time approximation:} 
    The time interval $[0,T]$ is partitioned into (uniformly) divided sub-intervals with step size $\tau>0$, and the time mesh $\Itau:=\{0=t^0,t^1=\tau,\ldots, t^{n_T}=T\}$ is introduced, where $n_T=\tfrac{T}{\tau}$ denotes the total number of time steps. We introduce the discrete time derivative and the mid point approximation via
            \begin{equation*}
                \partial_t^\tau g^{n+1} =  \frac{g^{n+1}-g^n}{\tau}, \qquad g^{n+1/2}= \frac{g^{n+1}+g^n}{2}.
            \end{equation*}
    Furthermore, we introduce the abbreviations 
    \begin{align*}
    X_{h}^{n+1/2}&:=X(\phi_{h}^{n+1/2},\vtheta_{h}^{n+1}) \text{ for } X\in\{M,N,K,\bm{\sigma}\} \\
    \partial_{\phi} e(\phi_h,s_h)&:=\frac{1}{\tau}\int_{t^n}^{t^{n+1}}  \partial_{\phi}e(I_1^\tau\phi_h(\ell),I_1^\tau s_h(\ell)) \,\mathrm{d}\ell, & I_1^\tau g(\ell) &:= \frac{g^{n+1}-g^n}{\tau}(\ell-t^n) + g^{n}\\
    \partial_{\nabla\phi} e(\phi_h,s_h) &:= \frac{1}{\tau}\int_{t^n}^{t^{n+1}}  \partial_{\nabla\phi}e(I_1^\tau\phi_h(\ell),I_1^\tau s_h(\ell)) \,\mathrm{d}\ell, & \partial_{s} e(\phi_h,s_h) &:=\frac{1}{\tau}\int_{t^n}^{t^{n+1}}  \partial_{s}e(I_1^\tau\phi_h(\ell),I_1^\tau s_h(\ell)) \,\mathrm{d}\ell.
    \end{align*}

    \begin{problem}\label{prob:scheme}
    Given the initial data $(\phi_{h}^0,s_h^0,\u_{h}^0)\in \Vh\times\Vh\times\Xh$. Find $(\phi_h^{n+1},\mu_h^{n+1},s_h^{n+1},\vtheta_h^{n+1},\u_h^{n+1},p_h^{n+1})\in  \Vh\times\Vh\times\Vh\times\Vh\times\Xh\times\Qh$  such that
    \begin{align}
       \la \dt^\tau\phi_h^{n+1},\psi \ra &- \la \phi_{h}^{n+1/2}\u_{h}^{n+1/2},\nabla\psi \ra + \la M_{h}^{n+1/2}\nabla \tfrac{\mu_{h}^{n+1}}{\vtheta_{h}^{n+1}} ,\nabla\psi \ra + \la N_{h}^{n+1/2}\tfrac{\mu_{h}^{n+1}}{\vtheta_{h}^{n+1}} ,\psi \ra= 0, \label{eq:schemeB1}
        \\
        \la \mu_{h}^{n+1},\zeta \ra & - \la \partial_{\phi}e(\phi_h,s_h),\zeta\ra - \la \partial_{\nabla\phi} e(\phi_h,s_h),\tfrac{1}{\vtheta_h^{n+1}}\nabla(\zeta\vtheta_h^{n+1}) \ra = 0, \label{eq:schemeB2}
        \\
        \la \dt^\tau s_h^{n+1},\omega \ra &- \la s_{h}^{n+1/2}\u_{h}^{n+1/2},\nabla\omega \ra - \la \eta(\phi_{h}^{n+1/2},\vtheta_{h}^{n+1})\snorm{\Du_{h}^{n+1/2}}^2,\tfrac{\omega}{\vtheta_{h}^{n+1}} \ra \notag
        \\
        &- \la K_{h}^{n+1/2}\nabla\tfrac{1}{\vtheta_{h}^{n+1}},\nabla\tfrac{\omega}{\vtheta_{h}^{n+1}}\ra - \la M_h^{n+1/2}\nabla\tfrac{\mu_{h}^{n+1}}{\vtheta_{h}^{n+1}},\nabla\tfrac{\omega\mu_{h}^{n+1}}{\vtheta_{h}^{n+1}} \ra - \la N_h^{n+1/2}\tfrac{\mu_{h}^{n+1}}{\vtheta_{h}^{n+1}},\tfrac{\omega\mu_{h}^{n+1}}{\vtheta_{h}^{n+1}} \ra \notag\\
        & - \la \dt^\tau\phi_h^{n+1}\partial_{\nabla\phi} e + \bm{\sigma}_{h}^{n+1/2}\cdot\u_{h}^{n+1/2},\tfrac{1}{\vtheta_{h}^{n+1}}\nabla\omega_{h,\tau} \ra= 0, \label{eq:schemeB3} 
        \\
        \la \vtheta_{h}^{n+1},\chi \ra &- \la \partial_{s} e(\phi_h,s_h),\chi  \ra= 0. \label{eq:schemeB4}
        \\
        \la \dt^\tau\u_h^{n+1},\vv \ra &+ \mathbf{c}_{skw}(\u_{h}^{n+1/2},\u_{h}^{n+1/2},\vv) + \la \eta(\phi_{h}^{n+1/2},\vtheta_{h}^{n+1})\Du_{h}^{n+1/2},\Dv \ra - \la p_{h}^{n+1},\div(\vv) \ra \notag
        \\
        &+ \la \phi_{h}^{n+1/2}\nabla\mu_{h}^{n+1}+ \left(\tfrac{1}{\vtheta_{h}^{n+1}}\bm{\sigma}_{h}^{n+1/2} +  s_{h}^{n+1/2}\right)\nabla\vtheta_{h}^{n+1},\vv \ra = 0 \label{eq:schemeB5},
        \\
        \la \div(\u_{h}^{n+1/2}),q \ra &= 0,  \label{eq:schemeB6}
    \end{align}
    for all $(\psi,\xi,\omega,\chi,\vv,q)\in\Vh\times\Vh\times\Vh\times\Vh\times\Xh\times\Qh.$
    \end{problem}
    \begin{theorem}
        Given solutions $(\phi_{h}^{n+1},\mu_{h}^{n+1},s_{h}^{n+1},\vtheta_{h}^{n+1},\u_{h}^{n+1},p_{h}^{n+1})$ with $\vtheta_h^k>0$ then we have the following properties
        \begin{align*}
           \la \phi_{h}^{n+1}, 1\ra &= \la \phi_{h}^{0}, 1\ra  \text{ (if applicable) }\qquad \la s_{h}^{n+1}, 1\ra = \la s_{h}^{0}, 1\ra + \tau\sum_{k=1}^{n_T-1}\mathcal{P}(\mu_{h}^{k+1}, \vtheta_{h}^{k+1},\u_{h}^{k+1}), 
           \\
           &\la e(\phi_h^{n+1},\nabla\phi_h^{n+1},\vtheta_h^{n+1}) + \tfrac{1}{2}\snorm{\u_{h}^{n+1}}^2, 1\ra = \la e(\phi_h^{0},\nabla\phi_h^{0},\vtheta_h^{0})+ \tfrac{1}{2}\snorm{\u_{h}^{0}}^2, 1\ra.
        \end{align*}
    
    \end{theorem}
    
    \begin{proof}
     Conservation of mass follows directly by using $\psi=1$ in \eqref{eq:schemeB1}. Entropy production is obtained by inserting $\omega=1$ into \eqref{eq:schemeB3}. Total energy conservation is obtained as follows
     \begin{align*}
    \frac{1}{\tau}&\la e_h^{n+1} - e_h^n,1 \ra = \frac{1}{\tau}\int_{t^n}^{t^{n+1}} \la \dt e(I_1^\tau\phi,I_1^\tau s), 1\ra \\
    &= \frac{1}{\tau}\int_{t^n}^{t^{n+1}} \la \dt I_1^\tau s(t), \partial_s e(I_1^\tau\phi,I_1^\tau s) \ra + \la \dt I_1^\tau \phi, \partial_\phi e(I_1^\tau\phi,I_1^\tau s) \ra + \la \dt I_1^\tau \nabla\phi, \partial_{\nabla\phi} e(I_1^\tau\phi,I_1^\tau s) \ra
    \\
    &= \int_{t^n}^{t^{n+1}} \la \dt^\tau s_h^{n+1},\partial_{s} e(\phi_h,s_h)\ra + \la\dt^\tau \phi_h^{n+1},\partial_{\phi} e(\phi_h,s_h)\ra + \la\dt^\tau \nabla\phi_h^{n+1},\partial_{\nabla\phi} e(\phi_h,\nabla\phi_h)\ra.
    \end{align*}
    Here we used that $\partial_t I_1^\tau g = \dt^\tau g_h^{n+1}$ on $[t^n,t^{n+1}]$. At this point we perform the same testing procedure as in Lemma \ref{lem:var} to obtain total energy conservation.
    \end{proof}

    \begin{remark}
        The variational formulation and the resulting numerical method can be generalised accounting for cross-coupling terms between the phase-field and the temperature equation, i.e. effects related to Soret and Dufour effects, cf. \cite{Hoehn25}. Furthermore, the mobility and heat conductivity can be generalised to matrix-valued functions and also the possible cross-coupling terms to matrix- or vector-valued functions. Furthermore, also without further problems to extension to different boundary conditions as in \cite{Hoehn26} is possible.
    \end{remark}

\section{Numerical example}
    In this section we will first perform a convergence analysis and afterwards consider application-inspired examples similar to \cite{Hoehn26}.
    The scheme described in the previous section, namely \eqref{eq:schemeB1}--\eqref{eq:schemeB6}, is implemented in NGSolve \cite{schoberl2014c++}. The non-linear system is solved using Newton's method with an absolute tolerance of $10^{-8}$ for the $L^2$ norm of the updates. The computational domain is the two-dimensional region $\Omega = (0,1)^2$, where $(x,y)$ denotes a point in $ \Omega$.
    
    To this end we consider two generic Helmholtz free energies of the form
    \begin{align*}
        F_i(\phi,\nabla\phi,\vtheta) = \frac{\kappa_i(\vtheta)}{2}\snorm{\nabla\phi}^2 - \vtheta(\log(\vtheta)-1) + \vtheta f_1(\phi) + f_2(\phi)
    \end{align*}
    for the to generic choices $\kappa_A(\vtheta)=\kappa_A$ and $\kappa_B(\vtheta)=\kappa_B\vtheta$. In both cases we can compute entropy and energy as
    \begin{align*}
        s_A &= \log(\vtheta) - f_1(\phi), & s_B &= -\frac{\kappa_1}{2}\snorm{\nabla\phi}^2 + \log(\vtheta) - f_1(\phi), \\
        e_A &= \vtheta + f_2(\phi) + \frac{\kappa_0}{2}\snorm{\nabla\phi}^2, & e_B &= \vtheta + f_2(\phi),
    \end{align*}  
    This can be inverted by standard algebraic manipulations and yields
    \begin{align*}
    e_A(\phi,\nabla\phi,s) := \exp\Big(s+f_1(\phi)\Big) + f_2(\phi) + \frac{\kappa_0}{2}\snorm{\nabla\phi}^2, \quad  e_B(\phi,\nabla\phi,s) := \exp\left(\frac{\kappa_1}{2}\snorm{\nabla\phi}^2+s+f_1(\phi)\right) + f_2(\phi).    
    \end{align*}
    In accordance with \cite{Hoehn26} with reference temperature $1$ we choose
    \begin{equation*}
     f_1(\phi) =  -H_{cf}\phi^2(1-\phi)^2 - H(\phi), \quad f_2(\phi)= (H_{pt}+H_{cf})\phi^2(1-\phi)^2 +  H(\phi)  
    \end{equation*}
    with the interpolation function $H(\phi)=\phi^3(6\phi^2-15\phi+10)$, which is extended outside of $\phi\in[0,1]$ by zero or one. The viscosity $\eta$ is defined as $\eta(\phi)=\eta_l\eta_s(H_\phi(\phi)(\eta_s-\eta_l)+\eta_l)^{-1}$
    where $\eta_l$ and $\eta_s$ denote the viscosities of the liquid and solid phase, respectively. The remaining parameters are given as follows: $\M=0$, $\K=0.01$, $N=10$, $\eta_s=1$, $\eta_l=0.001$, $\kappa_A=\kappa_B=0.025^2$, $H_\mathrm{pt}=1$, $H_\mathrm{cf}=0.1.$

\subsection{Convergence test}
We only examine spatial convergence by fixing the time step size to $\tau=10^{-4}$ with an end time of $T=0.01$, and consider spatial resolutions of $h_k=2^{-2-k}$ for $k=0,\ldots,6$. The initial conditions are chosen as follows:
\begin{align*}
    \phi_0(x,y)&:=2^{-2}\left(\tanh\left(\frac{\sqrt{(x-0.75)^2+(y-0.75)^2}-0.15)}{3.25\cdot 10^{-2}}\right)\right.\\
    &\qquad-\left.\tanh\left(\frac{\sqrt{(x-0.25)^2+(y-0.25)^2}-0.15)}{3.25\cdot 10^{-2}}\right)+2\right),\\
    \vtheta_0(x,y)&:=\exp\left(\ln(0.5)\cdot0.5(\sin(4\pi x)\sin(4\pi y)+1)(\sin(2\pi x)+\sin(2\pi y))\right),\\
    s_0(x,y) &= s_i(\phi_0(x,y),\vtheta_0(x,y)), \text{ for } i\in\{A,B\},\\
    \u_0(x,y) &= 10^{-1}\big(-\sin(\pi x)^2\sin(2\pi y),\sin(2\pi x)\sin(\pi y)^2\big)^{\top}.
\end{align*}
        
The errors are divided into two sets, and since no analytical solution is available, they are computed with respect to the corresponding refined solution. The errors of the first set, related to the $L^\infty$ norm in time, are given as:
\begin{align*}
    \mathrm{err}(\nabla\phi,h_k)&:= \max_{n\in\Itau}\norm{\phi^n_{h_k}-\phi^n_{h_{k+1}}}_{H^1},
    &
    \mathrm{err}(s,h_k)&:=\max_{n\in\Itau}\norm{s^n_{h_k}-s^n_{h_{k+1}}}_{L^2},
    \\
    \mathrm{err}(\u,h_k)&:=\max_{n\in\Itau}\norm{\u^n_{h_k}-\u^n_{h_{k+1}}}_{L^2},
\end{align*}
while the errors of the second set, related to the $L^2$ norm in time, are defined as:
\begin{align*}
    \mathrm{err}(\mu,h_k)&:=\sqrt{\tau\sum_{n=1}^{n_\tau}\norm{\mu^n_{h_k}-\mu^n_{h_{k+1}}}_{L^2}^2},
    &
    \mathrm{err}(\nabla\vtheta,h_k)&:=\sqrt{\tau\sum_{n=1}^{n_\tau}\norm{\vtheta^n_{h_k}-\vtheta^n_{h_{k+1}}}_{H^1}^2}
    \\
    \mathrm{err}(\nabla\u,h_k)&:=\sqrt{\tau\sum_{n=1}^{n_\tau}\norm{\u^n_{h_k}-\u^n_{h_{k+1}}}_{H^1}^2},
    &
    \mathrm{err}(p,h_k)&:=\sqrt{\tau\sum_{n=1}^{n_\tau}\norm{p^n_{h_k}-p^n_{h_{k+1}}}_{L^2}^2}.
\end{align*}

\begin{table}[htbp!]
    \centering
    %\small
    \caption{Errors and rates in $L^\infty$ time norms for spatial convergence for $i=A$.\label{tab:1}} 
    %\resizebox{\columnwidth}{!}{
    \begin{tabular}{c|l|c|l|c|l|c}
           $k$ & $\mathrm{err}(\nabla\phi,h_k)$ & $\mathrm{eoc}_k$  & $\mathrm{err}(s,h_k)$ & $\mathrm{eoc}_k$  & $\mathrm{err}(\u,h_k)$ & $\mathrm{eoc}_k$ \\
           \hline
            0 & $8.51\cdot 10^{-1}$ & -- & $1.26\cdot 10^{-1}$ & -- & $4.76\cdot 10^{-3}$ & --\\
            1 & $1.25\cdot 10^{0}$ & $-0.56$ & $7.07\cdot 10^{-2}$ & $0.83$ & $4.84\cdot 10^{-3}$ &
            $-0.03$\\
            2 & $6.85\cdot 10^{-1}$ & $0.87$ & $2.11\cdot 10^{-2}$ & $1.75$ & $2.70\cdot 10^{-3}$ &
            $0.84$\\
            3 & $4.02\cdot 10^{-1}$ & $0.77$ & $7.73\cdot 10^{-3}$ & $1.45$ & $8.11\cdot 10^{-4}$ &
            $1.74$\\
            4 & $2.08\cdot 10^{-1}$ & $0.95$ & $1.96\cdot 10^{-3}$ & $1.98$ & $1.11\cdot 10^{-4}$ &
            $2.87$\\
            5 & $1.05\cdot 10^{-1}$ & $0.99$ & $5.00\cdot 10^{-4}$ & $1.97$ & $1.90\cdot 10^{-5}$ &
            $2.54$
    \end{tabular}%}
\end{table}

\begin{table}[htbp!]
    \centering
    %\small
    \caption{Errors and rates in $L^2$ time norms for spatial convergence for $i=A$.\label{tab:2}} 
    %\resizebox{\columnwidth}{!}{
    \begin{tabular}{c|l|c|l|c|l|c|l|c}
        $k$ & $\mathrm{err}(\mu,h_k)$ & $\mathrm{eoc}_k$  & $\mathrm{err}(\nabla\vtheta,h_k)$
         & $\mathrm{eoc}_k$  & $\mathrm{err}(\nabla\u,h_k)$ & $\mathrm{eoc}_k$  & $\mathrm{err}(p,h_k)$ & $\mathrm{eoc}_k$ \\
         \hline
        0 & $2.68\cdot 10^{-2}$ & -- & $2.95\cdot 10^{-1}$ & -- & $1.16\cdot 10^{-2}$ & -- & $6.
        05\cdot 10^{-3}$ & --\\
        1 & $1.03\cdot 10^{-2}$ & $1.37$ & $2.27\cdot 10^{-1}$ & $0.38$ & $1.48\cdot 10^{-2}$ &
        $-0.36$ & $7.48\cdot 10^{-3}$ & $-0.31$\\
        2 & $5.35\cdot 10^{-3}$ & $0.95$ & $1.17\cdot 10^{-1}$ & $0.95$ & $1.46\cdot 10^{-2}$ &
        $0.03$ & $2.69\cdot 10^{-3}$ & $1.47$\\
        3 & $1.68\cdot 10^{-3}$ & $1.67$ & $6.01\cdot 10^{-2}$ & $0.97$ & $9.67\cdot 10^{-3}$ &
        $0.59$ & $7.85\cdot 10^{-4}$ & $1.78$\\
        4 & $4.00\cdot 10^{-4}$ & $2.07$ & $2.94\cdot 10^{-2}$ & $1.03$ & $3.26\cdot 10^{-3}$ &
        $1.57$ & $1.94\cdot 10^{-4}$ & $2.02$\\
        5 & $1.13\cdot 10^{-4}$ & $1.83$ & $1.45\cdot 10^{-2}$ & $1.02$ & $5.81\cdot 10^{-4}$ &
        $2.49$ & $5.44\cdot 10^{-5}$ & $1.83$
    \end{tabular}%}
\end{table}

\begin{table}[htbp!]
    \centering
    %\small
    \caption{Errors and rates in $L^\infty$ time norms for spatial convergence for $i=B$.\label{tab:3}} 
    %\resizebox{\columnwidth}{!}{
    \begin{tabular}{c|l|c|l|c|l|c|l|c}
        $k$ & $\mathrm{err}(\nabla\phi,h_k)$ & $\mathrm{eoc}_k$  & $\mathrm{err}(s,h_k)$ & $\mathrm{eoc}_k$  & $\mathrm{err}(\u,h_k)$ & $\mathrm{eoc}_k$ \\
           \hline
        0 & $8.60\cdot 10^{-1}$ & -- & $1.26\cdot 10^{-1}$ & -- & $4.47\cdot 10^{-3}$ & --\\
        1 & $1.26\cdot 10^{0}$ & $-0.55$ & $7.09\cdot 10^{-2}$ & $0.83$ & $4.53\cdot 10^{-3}$ &
        $-0.02$\\
        2 & $6.86\cdot 10^{-1}$ & $0.88$ & $2.16\cdot 10^{-2}$ & $1.71$ & $2.35\cdot 10^{-3}$ &
        $0.95$\\
        3 & $4.03\cdot 10^{-1}$ & $0.77$ & $8.63\cdot 10^{-3}$ & $1.33$ & $7.35\cdot 10^{-4}$ &
        $1.67$\\
        4 & $2.08\cdot 10^{-1}$ & $0.95$ & $2.56\cdot 10^{-3}$ & $1.75$ & $1.21\cdot 10^{-4}$ &
        $2.60$\\
        5 & $1.05\cdot 10^{-1}$ & $0.99$ & $6.97\cdot 10^{-4}$ & $1.88$ & $2.20\cdot 10^{-5}$ &
        $2.46$
    \end{tabular}%}
\end{table}

\begin{table}[htbp!]
    \centering
    %\small
    \caption{Errors and rates in $L^2$ time norms for spatial convergence for $i=B$.\label{tab:4}} 
    %\resizebox{\columnwidth}{!}{
    \begin{tabular}{c|l|c|l|c|l|c|l|c|l|c}
        $k$ & $\mathrm{err}(\mu,h_k)$ & $\mathrm{eoc}_k$  & $\mathrm{err}(\nabla\vtheta,h_k)$
         & $\mathrm{eoc}_k$  & $\mathrm{err}(\nabla\u,h_k)$ & $\mathrm{eoc}_k$  & $\mathrm{err}(p,h_k)$ & $\mathrm{eoc}_k$ \\
         \hline
        0 & $2.53\cdot 10^{-2}$ & -- & $2.96\cdot 10^{-1}$ & -- & $1.10\cdot 10^{-2}$ & -- & $6.
        83\cdot 10^{-3}$ & --\\
        1 & $9.92\cdot 10^{-3}$ & $1.35$ & $2.28\cdot 10^{-1}$ & $0.38$ & $1.30\cdot 10^{-2}$ &
        $-0.25$ & $6.99\cdot 10^{-3}$ & $-0.03$\\
        2 & $5.26\cdot 10^{-3}$ & $0.92$ & $1.21\cdot 10^{-1}$ & $0.92$ & $1.30\cdot 10^{-2}$ &
        $0.00$ & $2.80\cdot 10^{-3}$ & $1.32$\\
        3 & $1.86\cdot 10^{-3}$ & $1.50$ & $6.41\cdot 10^{-2}$ & $0.91$ & $8.51\cdot 10^{-3}$ &
        $0.61$ & $8.74\cdot 10^{-4}$ & $1.68$\\
        4 & $4.53\cdot 10^{-4}$ & $2.04$ & $3.23\cdot 10^{-2}$ & $0.99$ & $3.30\cdot 10^{-3}$ &
        $1.37$ & $2.41\cdot 10^{-4}$ & $1.86$\\
        5 & $1.27\cdot 10^{-4}$ & $1.83$ & $1.52\cdot 10^{-2}$ & $1.09$ & $5.97\cdot 10^{-4}$ &
        $2.46$ & $7.09\cdot 10^{-5}$ & $1.76$
    \end{tabular}%}
\end{table}
The results for the test with $i=A$ can be found in Tables \ref{tab:1} and \ref{tab:2} while the results for $i=B$ can be found in Tables \ref{tab:3} and \ref{tab:4}, where the experimental order of convergence (eoc) is shown.
We observe an optimal or near-optimal spatial order in both cases. Only the orders of $\u_h,\mu_h,p_h$ are  below the expectation, which might be expected. For the velocity, the optimal order is three in $L^2$, the rates might be reduced due to the lower order coupling with the remaining piecewise linear discretisation, which also then reduces the rate of the pressure below two. For the chemical potential, the exaction in the $L^2$ norm is two. However, this variable appears in the entropy production as $\frac{\mu}{\vtheta}$, so the pure $L^2-$ norm is not a canonical energy norm.

\subsection{Application driven test}

We consider the temporal evolution of the above convergence test without inital flow field, which allows us to account for the effects of melting and solidification side by side, as in \cite{Hoehn26}. For the phase variable $\phi$, we consider a mixture of solid and melt around $\phi=0.5$ which contains a solidified phase in the right upper corner, i.e. $\phi=0$, and a melted region in the left lower corner, i.e. $\phi=1$. For $\vtheta$, we impose a profile that allows simultaneous overheating and undercooling in both regions, with values around the melting temperature, here $1$, in the rest of the domain. The temporal evolution of the volume fraction is depicted in Figure \ref{fig:ex_melt_phi} and the evolution of the temperature derivation, i.e. $\vtheta_h-1$, is visualised in Figure \ref{fig:ex_melt_theta}. The entropy evolution and the total energy error are visualised in Figure \ref{fig:ex_melt_struct} and demonstrate entropy production and total energy conservation, within Newton's tolerance. The temporal evolution for $i=A$ aligns with the results in \cite{Hoehn26}, and the evolution for $i=B$ is very similar, although a different free energy is used. We expect that the results of both simulations are comparable as the interface width is comparable for $\vtheta_h\approx 1$.

\begin{figure}[htbp!]
            \centering
            \begin{tabular}{c@{}c@{}c@{}c@{}c@{}}
                \multicolumn{5}{c}{\includegraphics[trim={40cm 66cm 40cm 0cm},clip,scale=0.15]{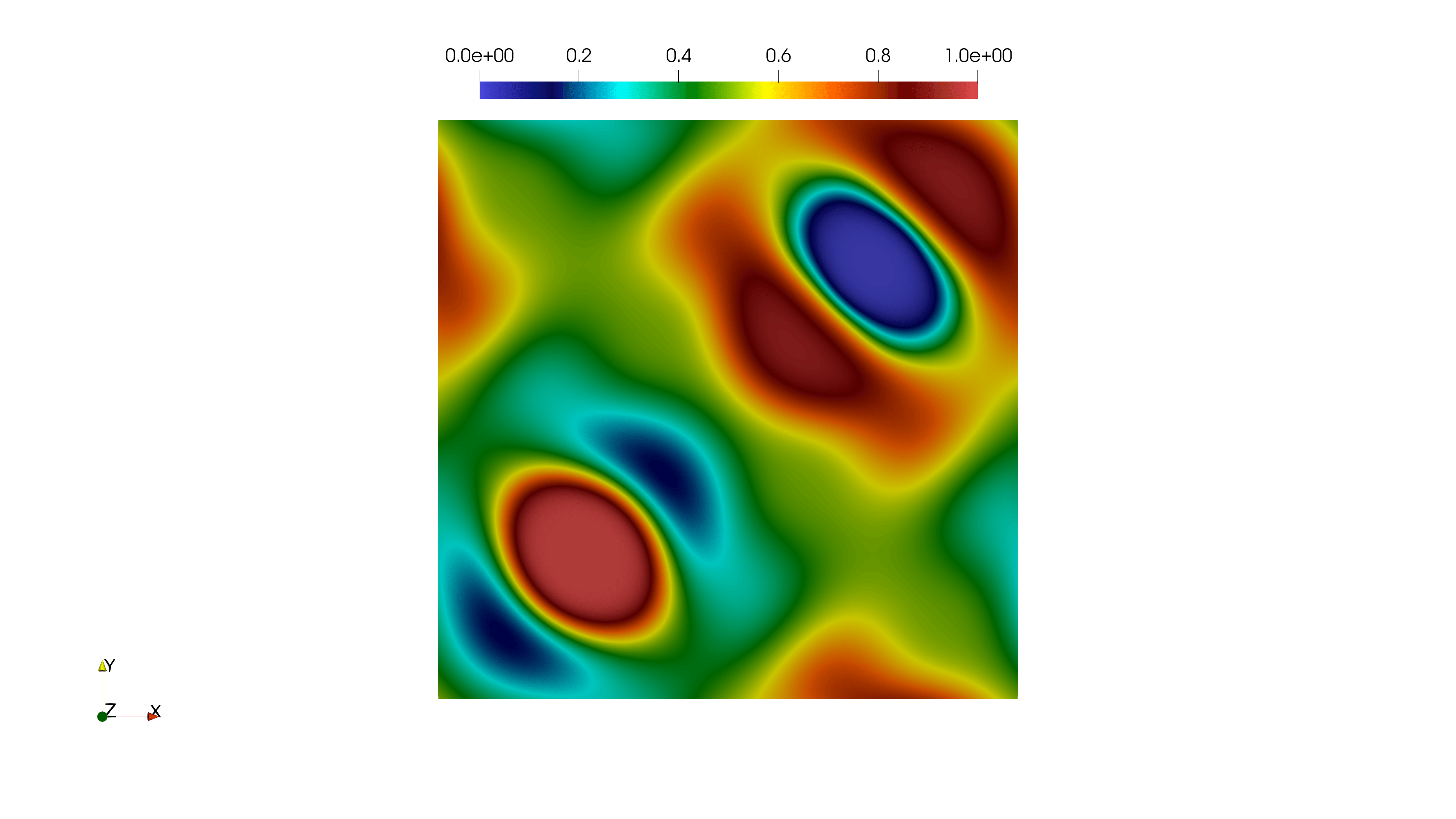}}\\
                \includegraphics[trim={40cm 10cm 40cm 10cm},clip,scale=0.06]{Bilder/phi_A.0010.png}
                &
                \includegraphics[trim={40cm 10cm 40cm 10cm},clip,scale=0.06]{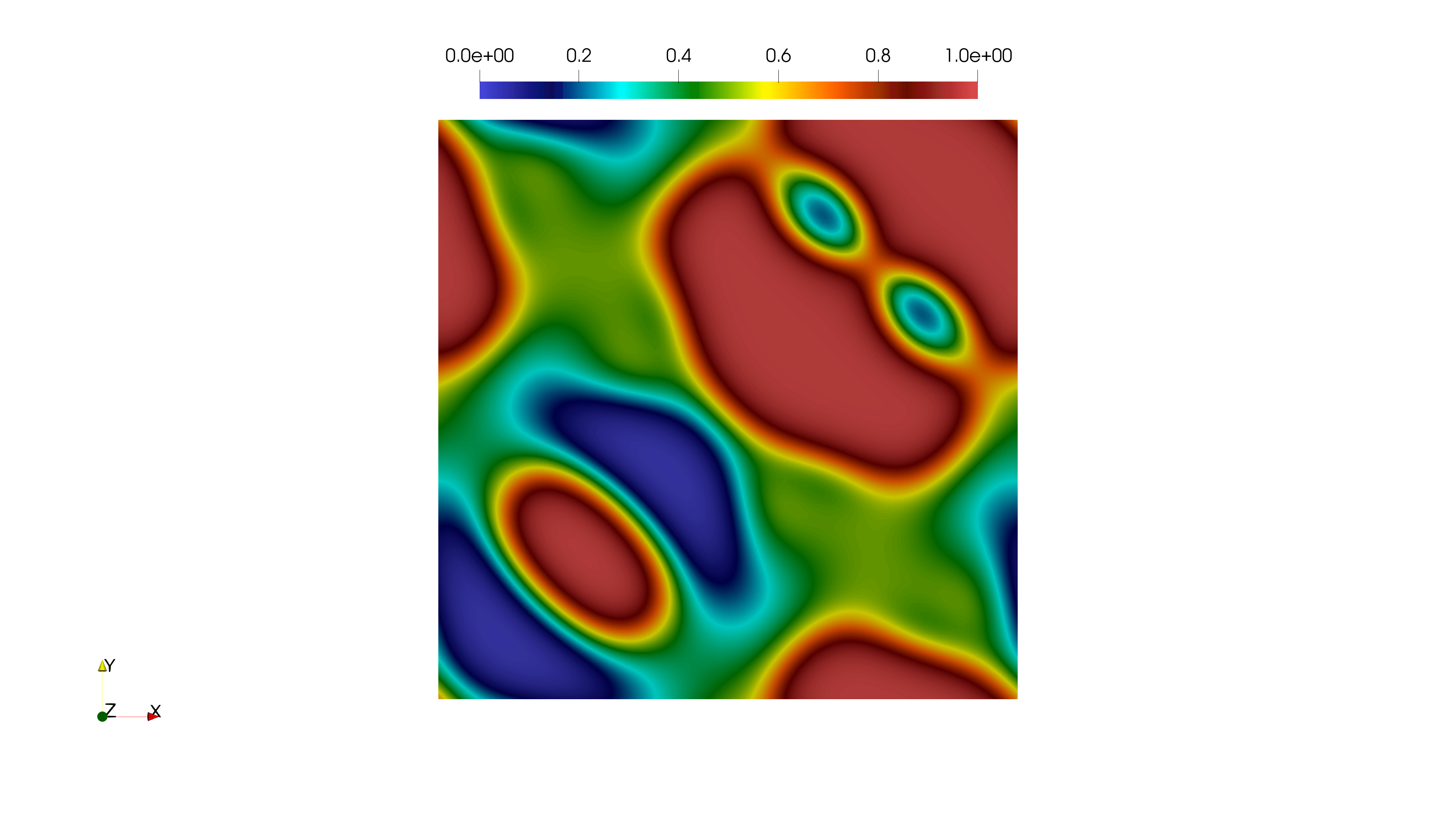} 
                &
                \includegraphics[trim={40cm 10cm 40cm 10cm},clip,scale=0.06]{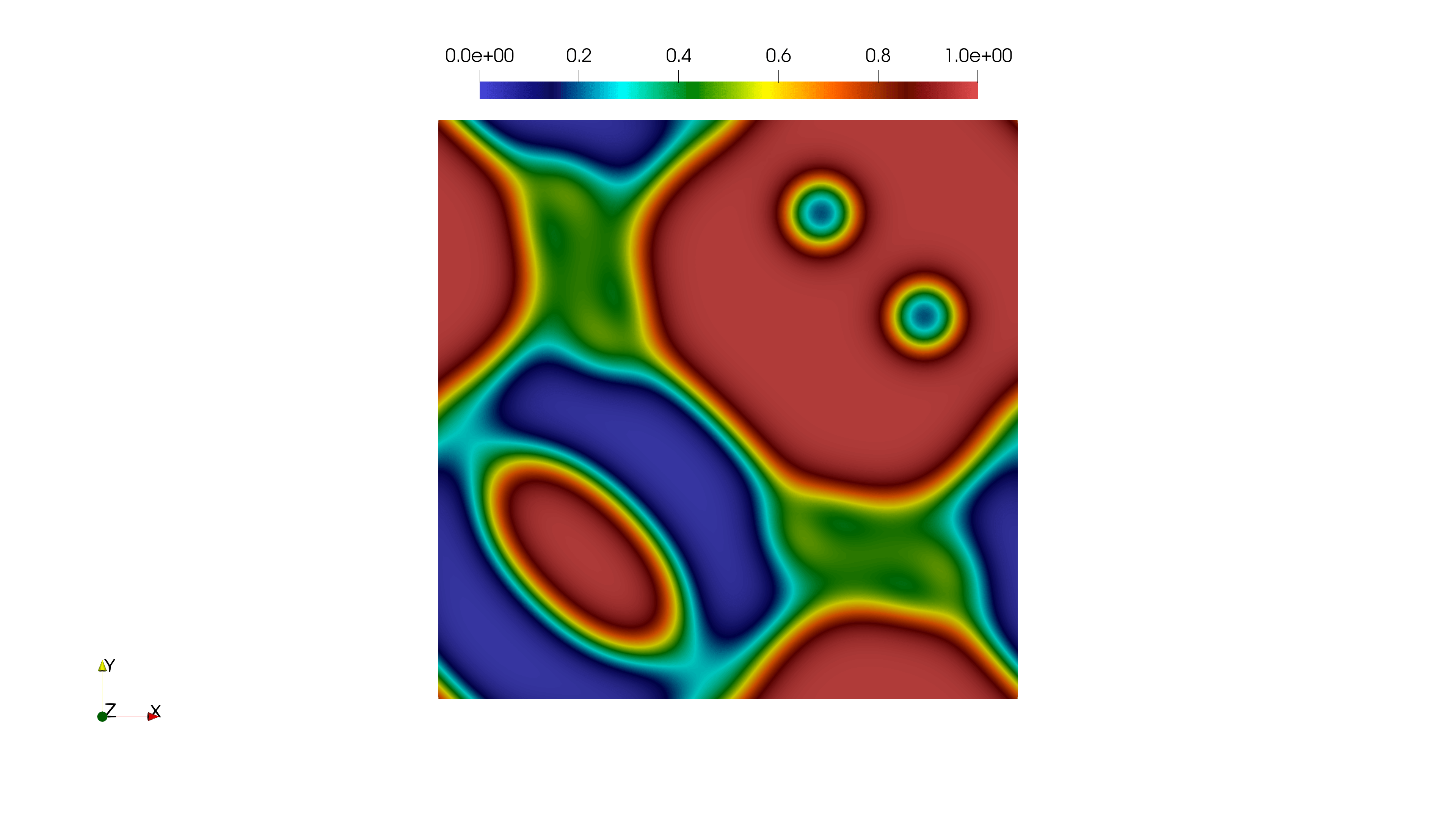}
                &
                \includegraphics[trim={40cm 10cm 40cm 10cm},clip,scale=0.06]{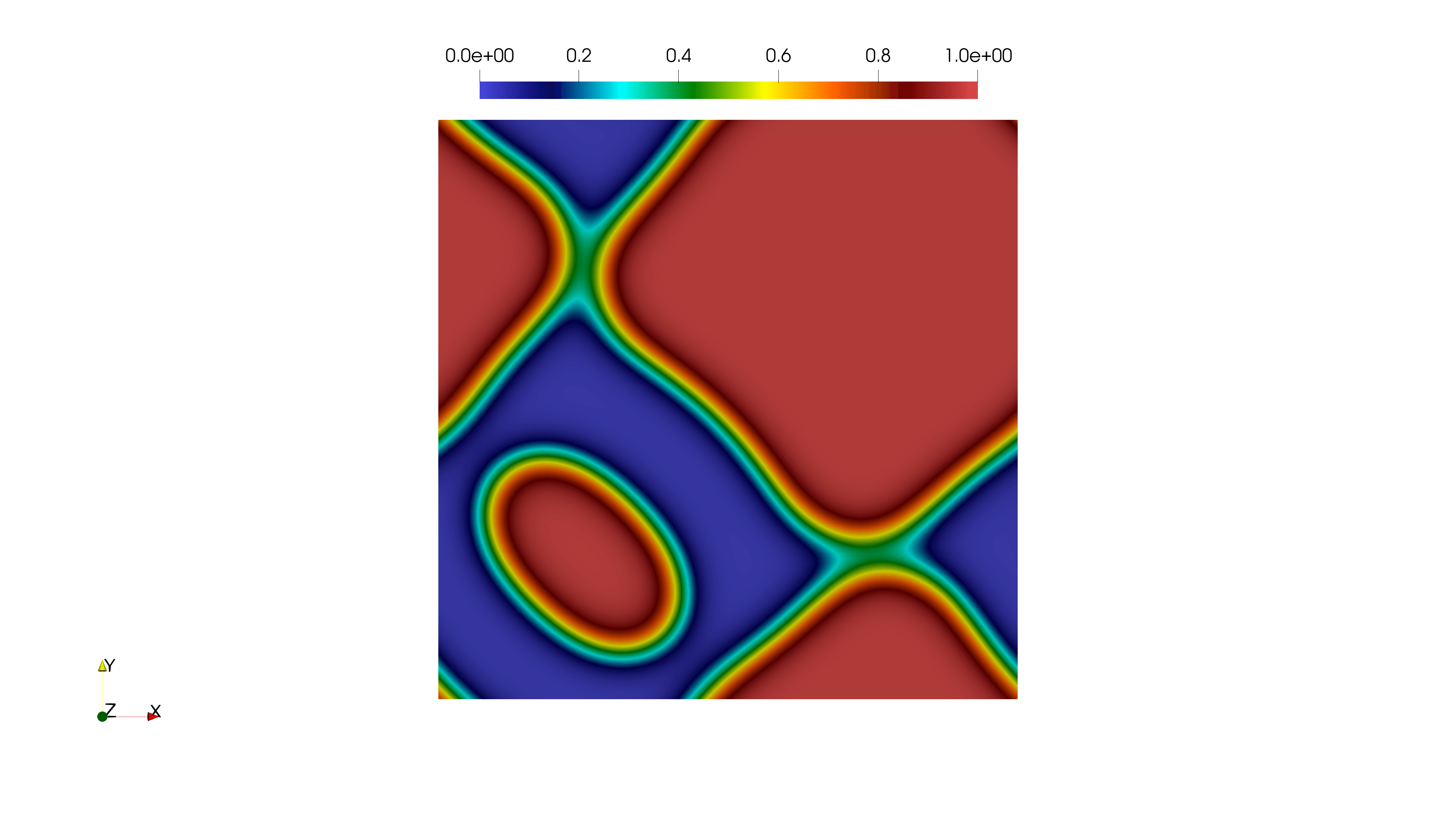} 
                &
                \includegraphics[trim={40cm 10cm 40cm 10cm},clip,scale=0.06]{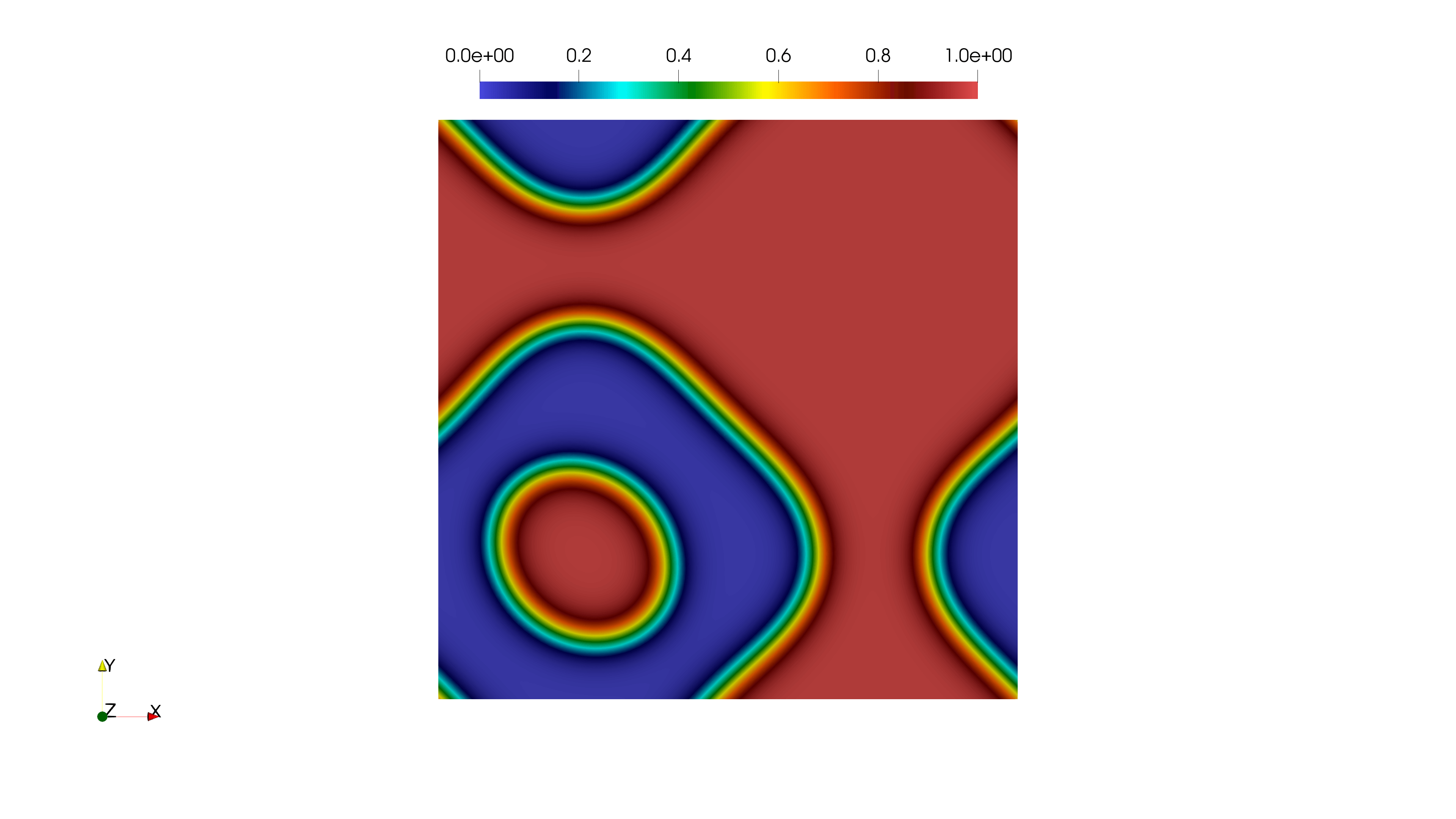} \\
                \includegraphics[trim={40cm 10cm 40cm 10cm},clip,scale=0.06]{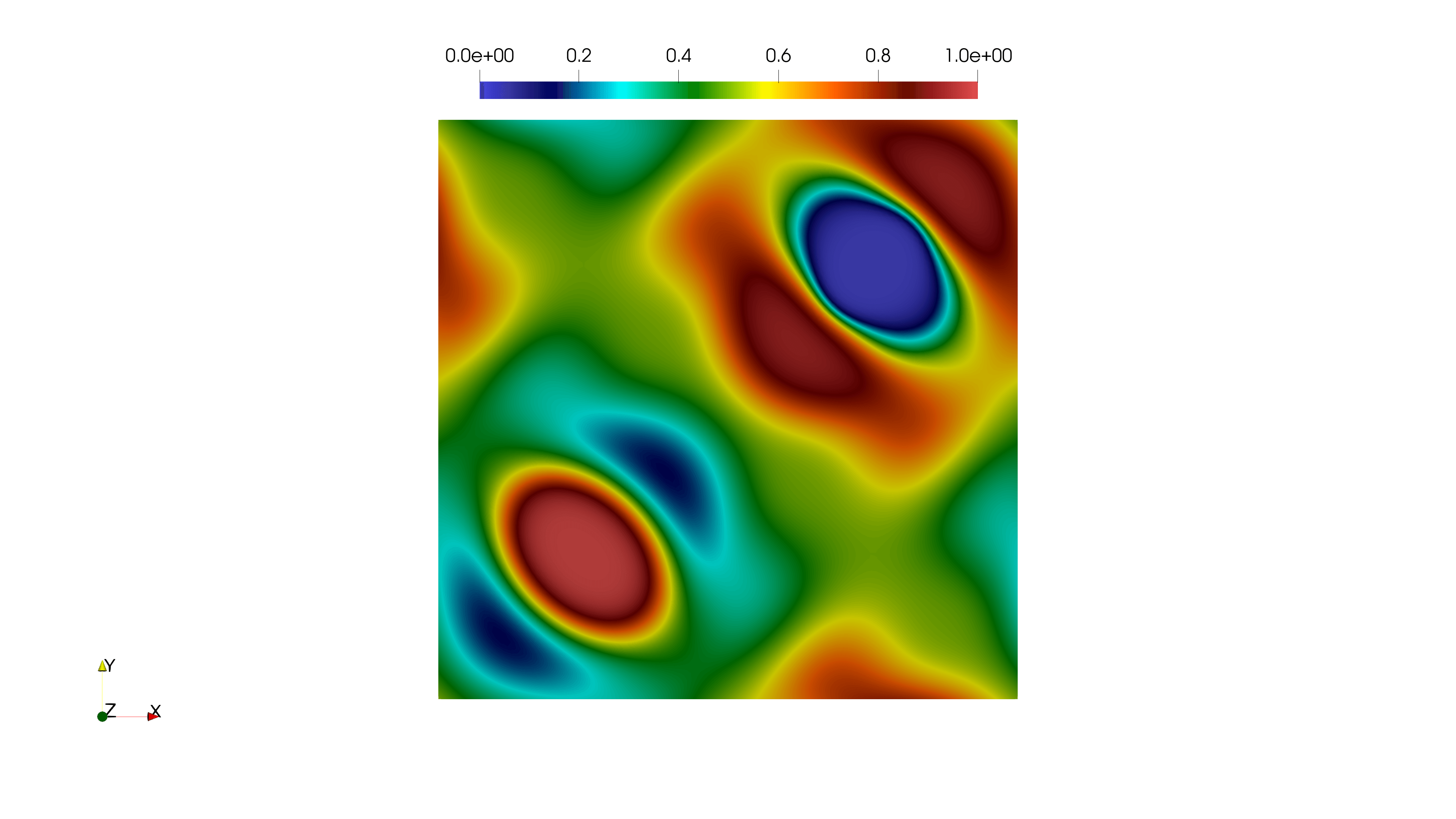}
                &
                \includegraphics[trim={40cm 10cm 40cm 10cm},clip,scale=0.06]{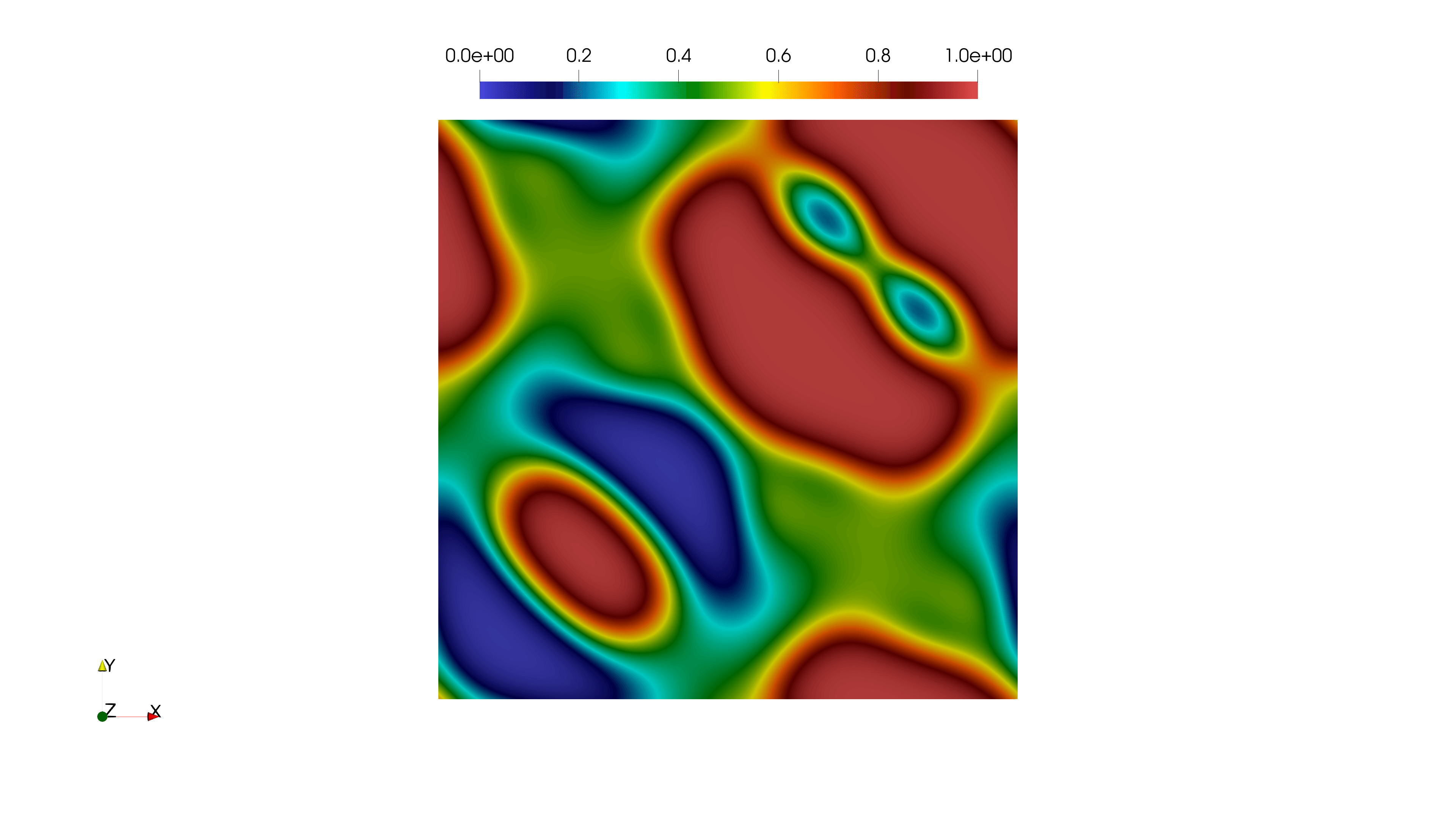} 
                &
                \includegraphics[trim={40cm 10cm 40cm 10cm},clip,scale=0.06]{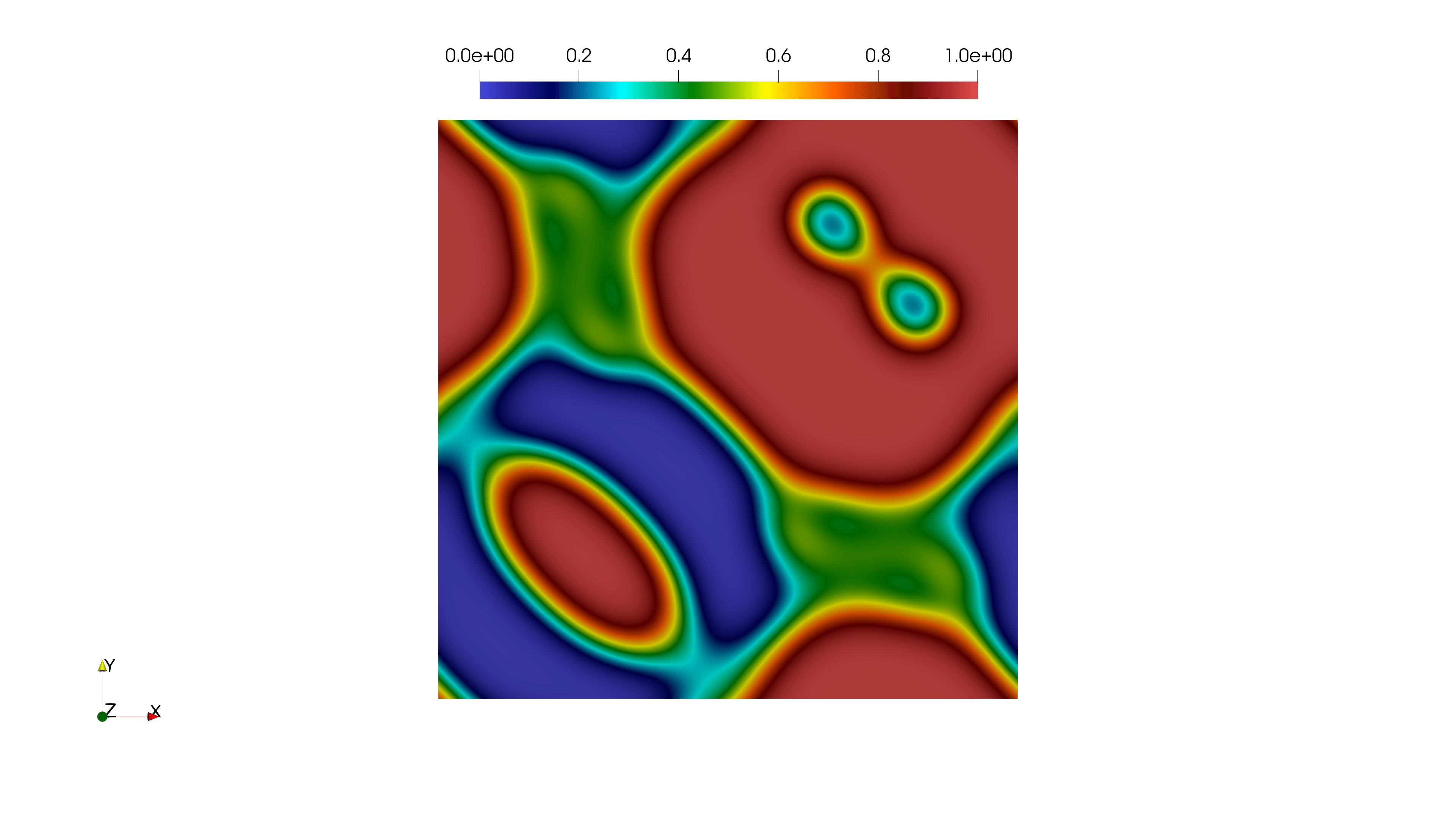}
                &
                \includegraphics[trim={40cm 10cm 40cm 10cm},clip,scale=0.06]{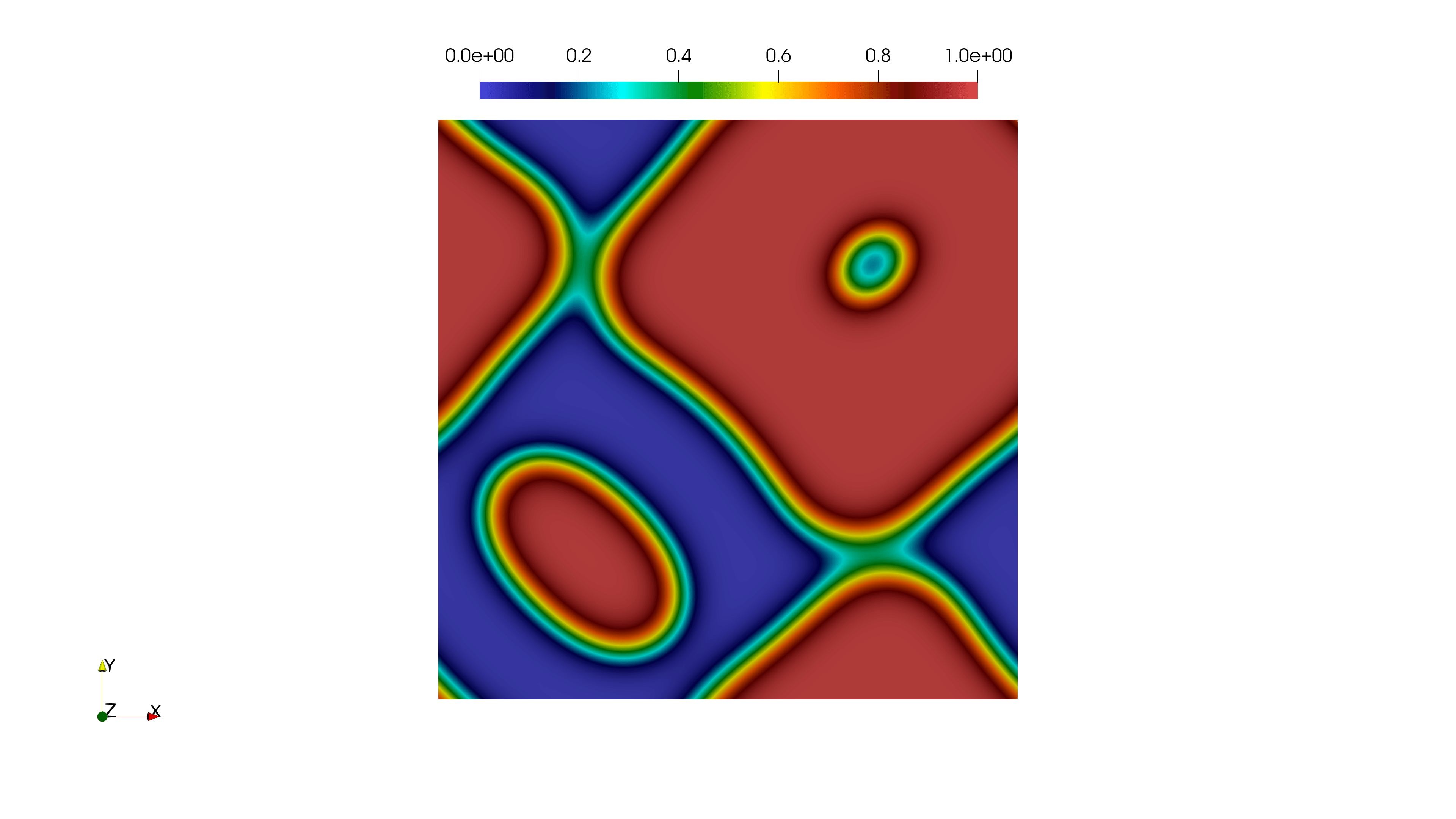} 
                &
                \includegraphics[trim={40cm 10cm 40cm 10cm},clip,scale=0.06]{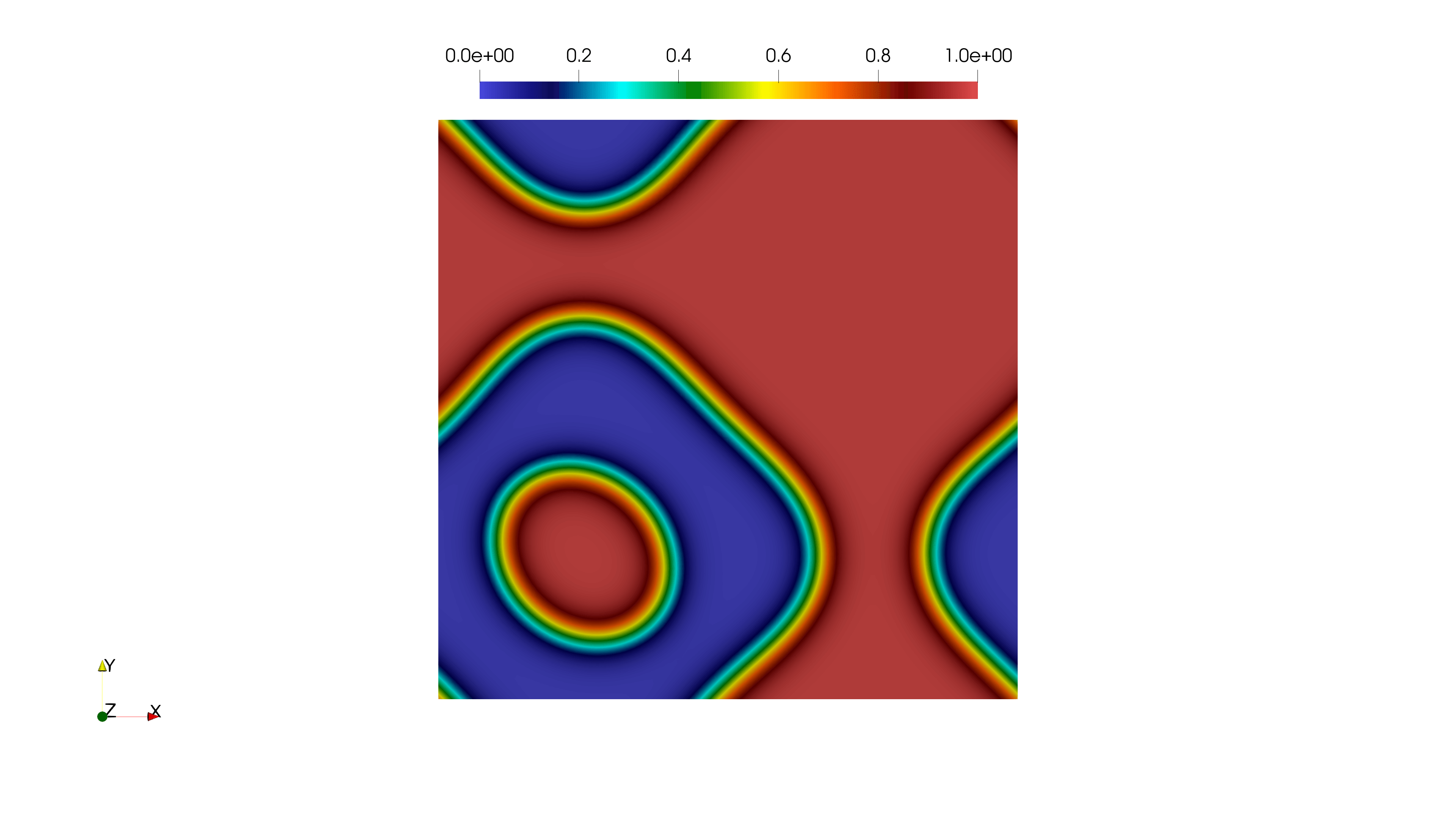} \\
                $t=0.1$ & $t=0.5$ & $t=1$ & $t=2$ & $t=5$  \\
            \end{tabular}
            \caption{Snapshots of the temporal evolution of $\phi_h$ for $i=A$ (top) compared with snapshots of $\phi_h$ for $i=B$ (bottom).}
            \label{fig:ex_melt_phi}
        \end{figure}

        \begin{figure}[htbp!]
            \centering
            %\hspace*{-1.5em}
            \begin{tabular}{c@{}c@{}c@{}c@{}c@{}}
            \multicolumn{5}{c}{\includegraphics[trim={40cm 66cm 40cm 0cm},clip,scale=0.15]{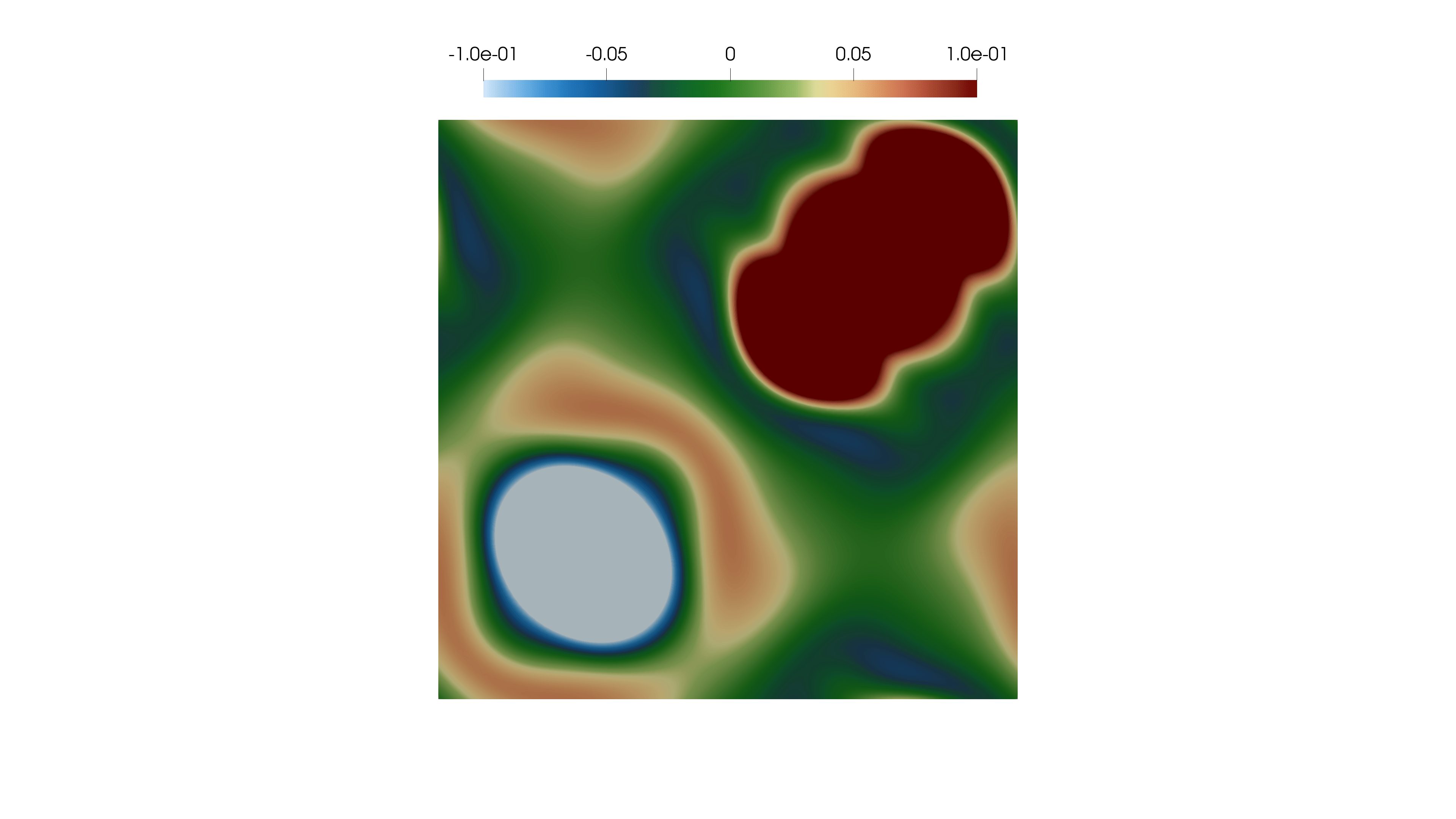}}\\
                \includegraphics[trim={40cm 10cm 40cm 10cm},clip,scale=0.06]{Bilder/theta_scale_a.0010.png}
                &
                \includegraphics[trim={40cm 10cm 40cm 10cm},clip,scale=0.06]{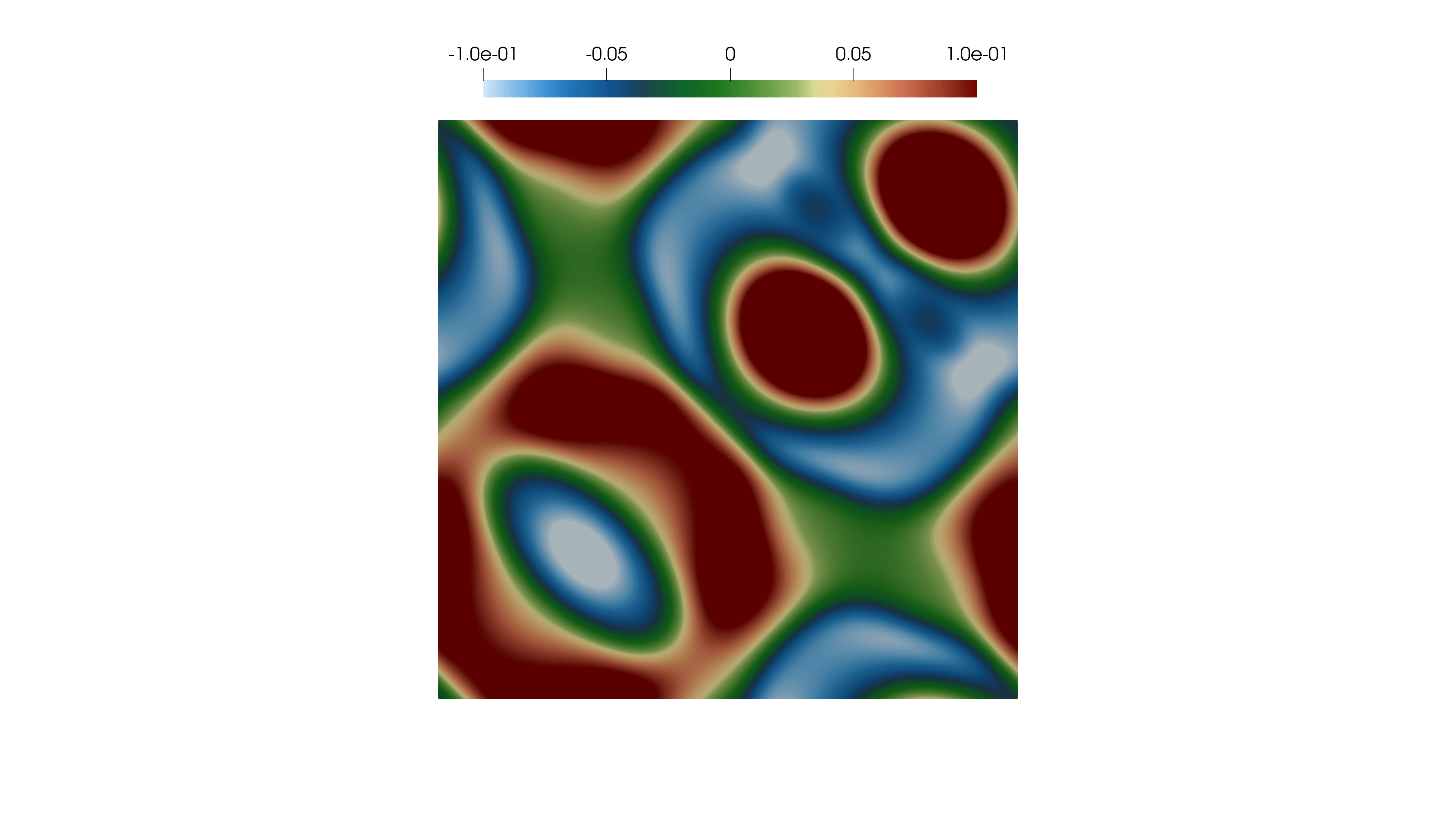} 
                &
                \includegraphics[trim={40cm 10cm 40cm 10cm},clip,scale=0.06]{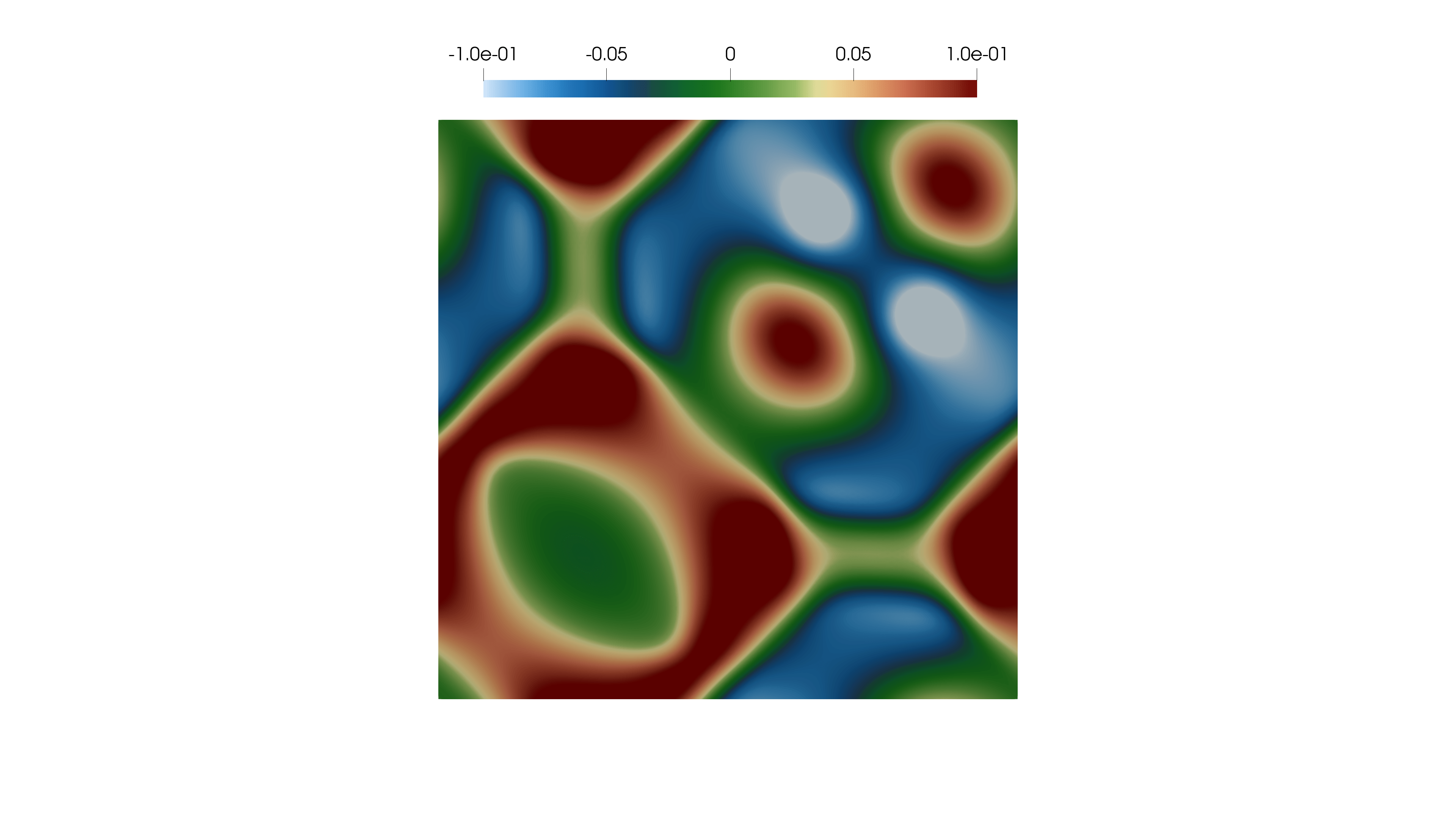}
                &
                \includegraphics[trim={40cm 10cm 40cm 10cm},clip,scale=0.06]{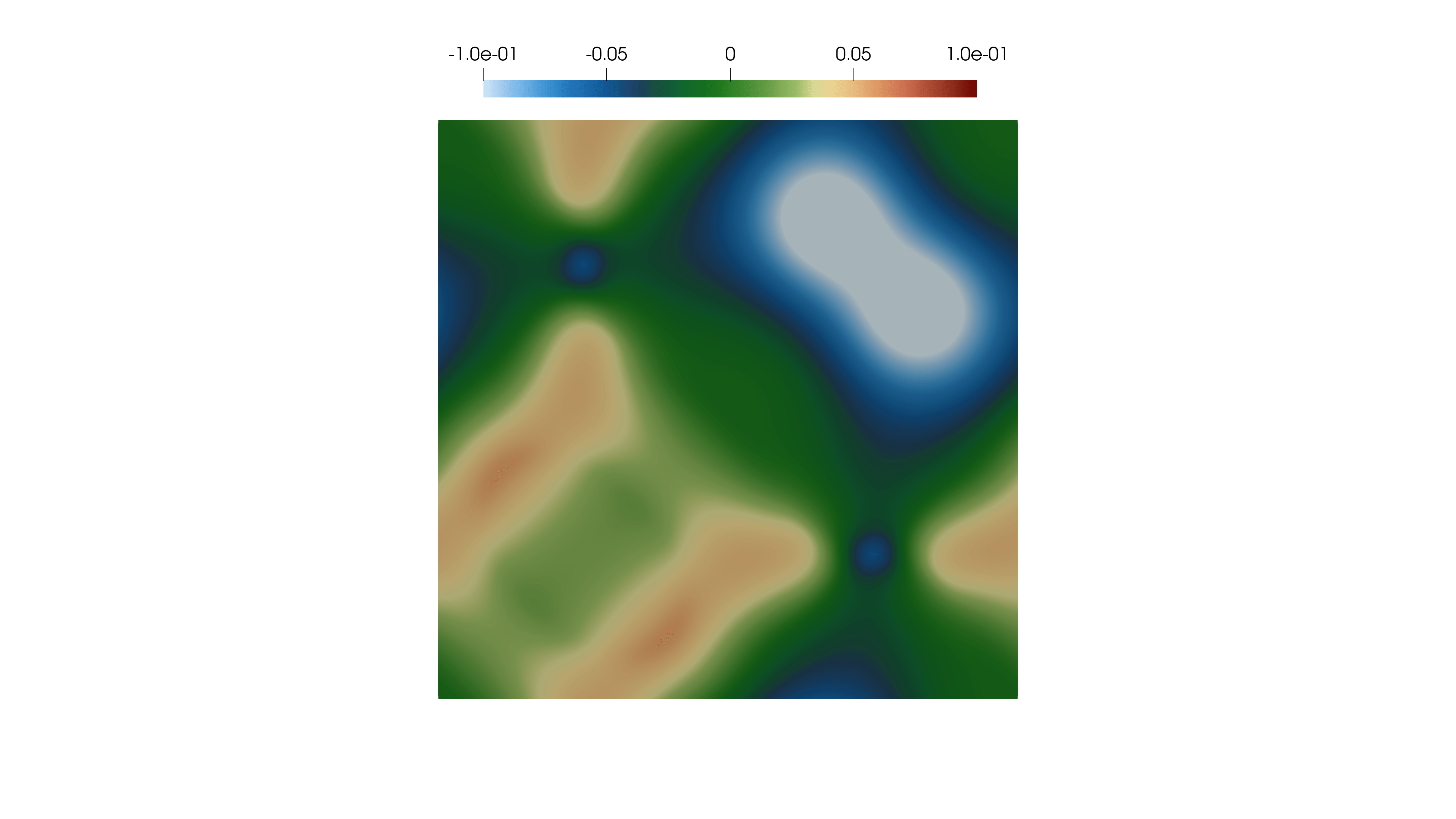} 
                &
                \includegraphics[trim={40cm 10cm 40cm 10cm},clip,scale=0.06]{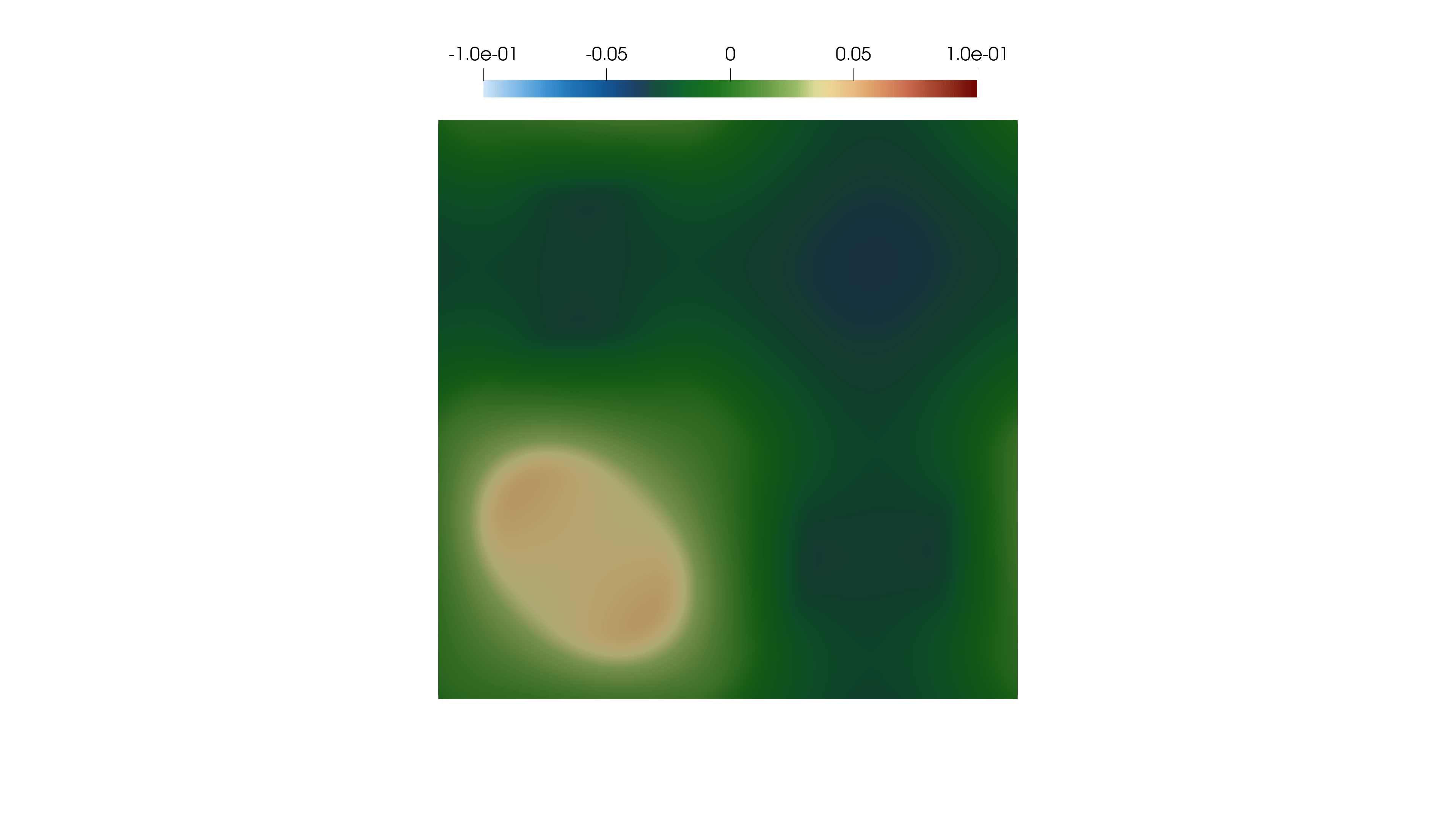} \\
                \includegraphics[trim={40cm 10cm 40cm 10cm},clip,scale=0.06]{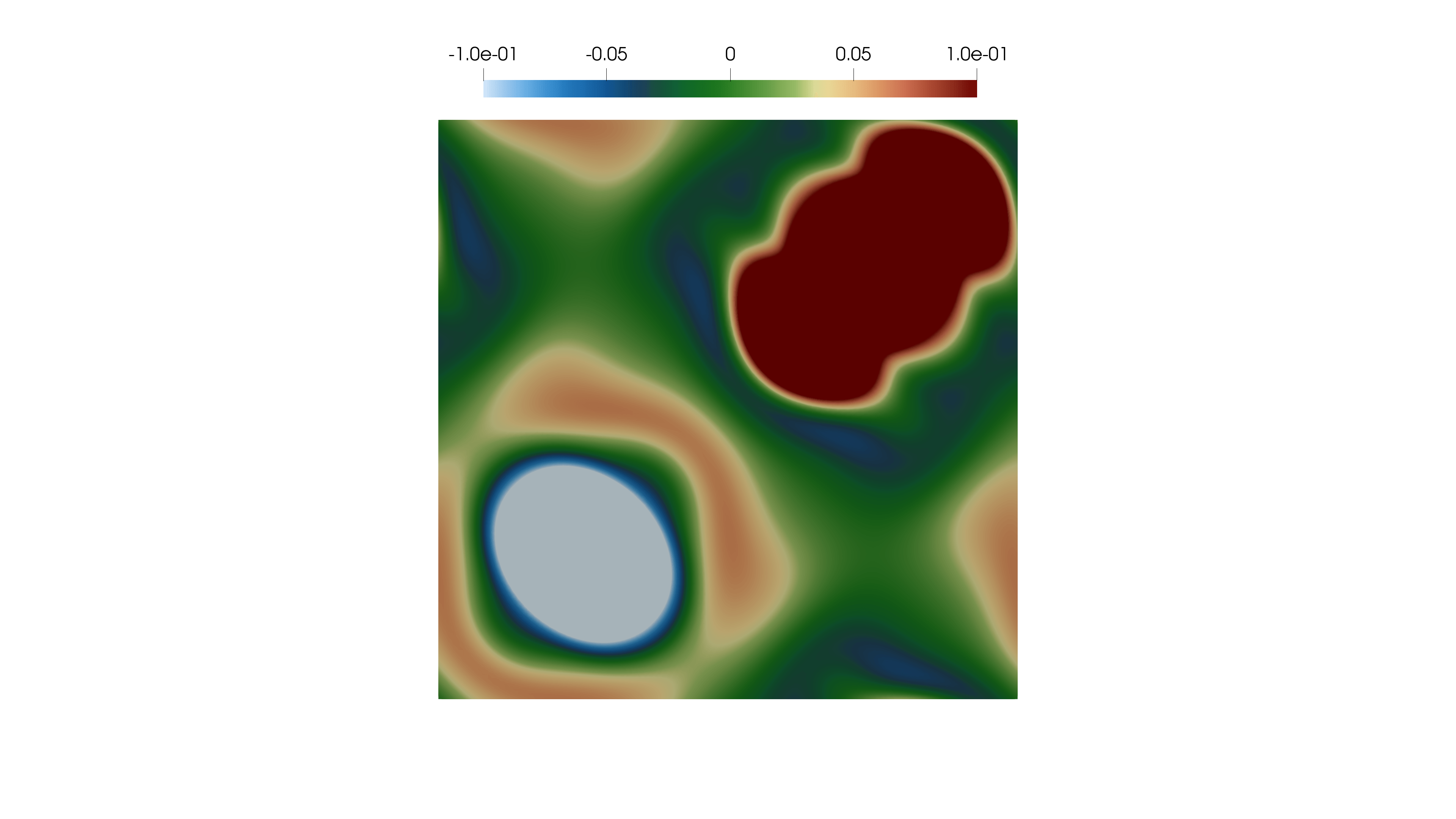}
                &
                \includegraphics[trim={40cm 10cm 40cm 10cm},clip,scale=0.06]{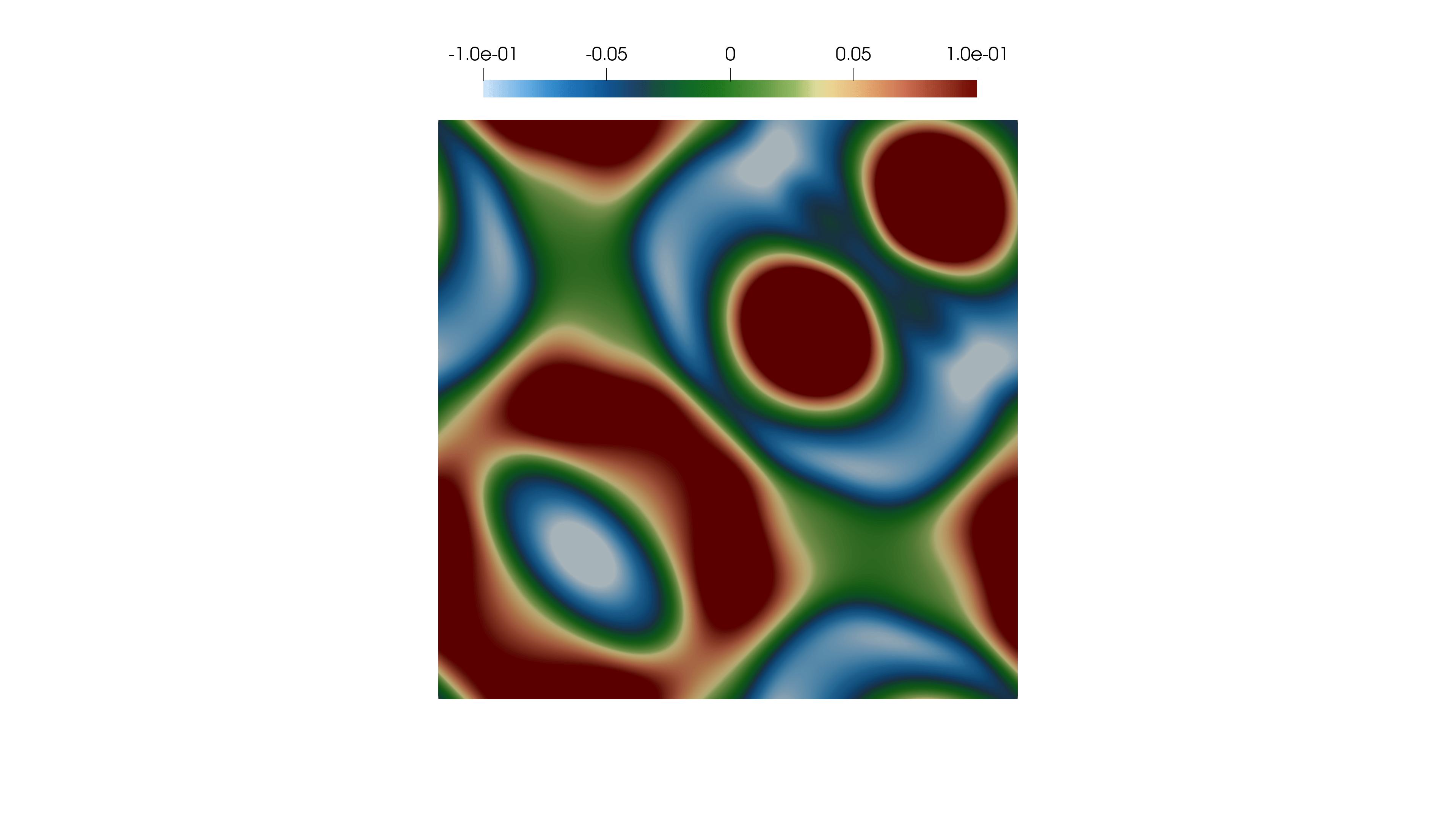} 
                &
                \includegraphics[trim={40cm 10cm 40cm 10cm},clip,scale=0.06]{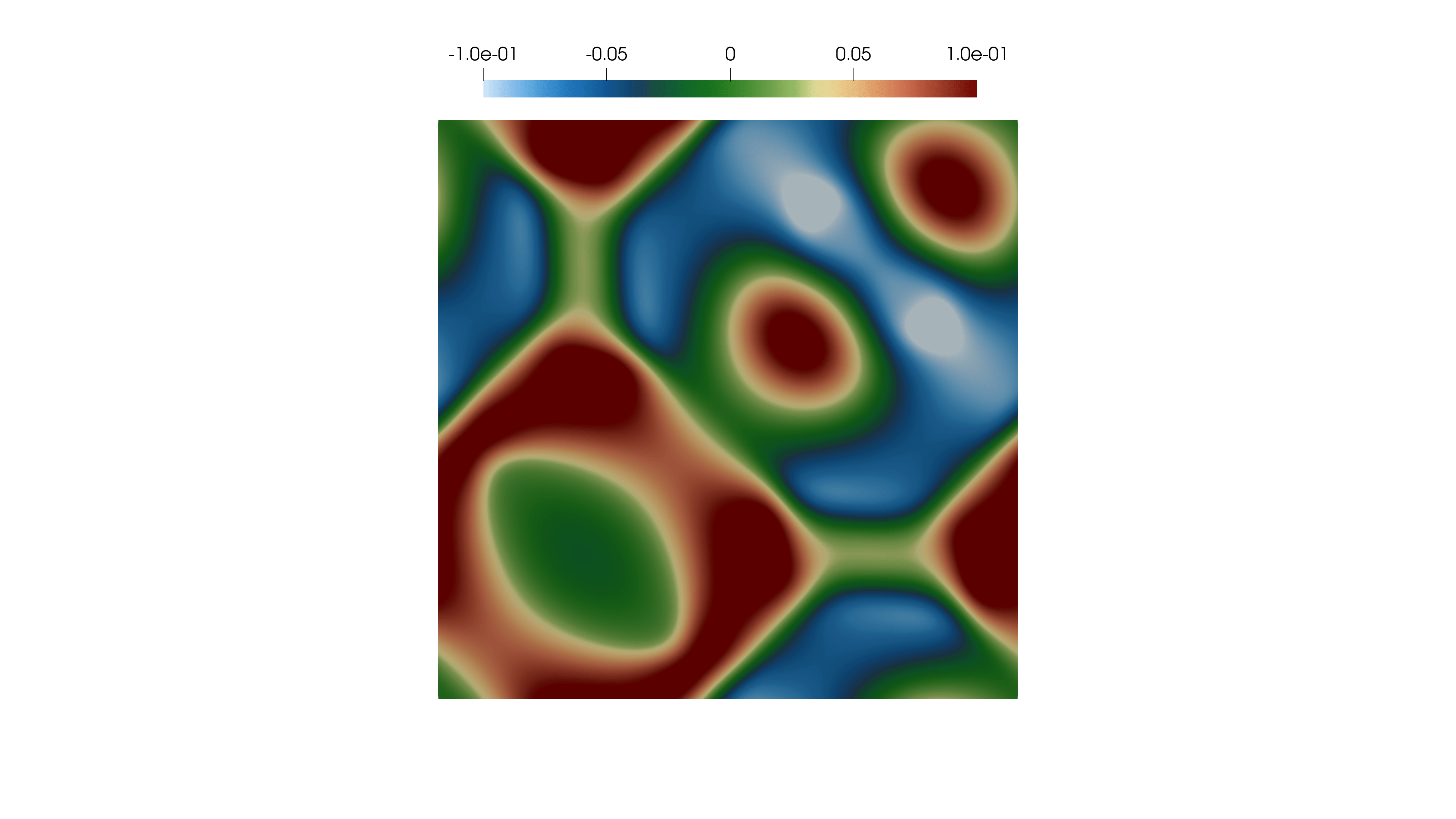}
                &
                \includegraphics[trim={40cm 10cm 40cm 10cm},clip,scale=0.06]{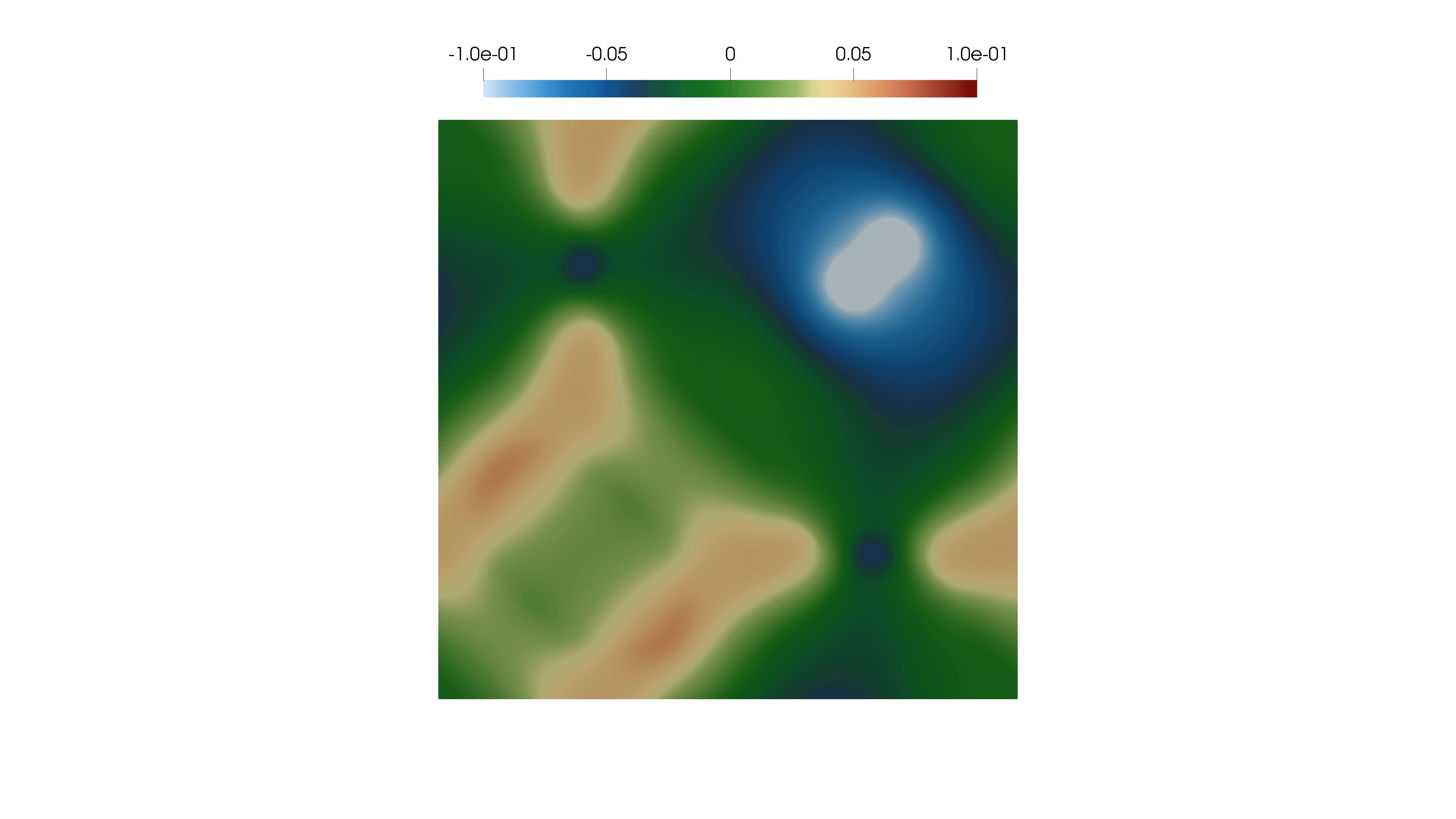} 
                &
                \includegraphics[trim={40cm 10cm 40cm 10cm},clip,scale=0.06]{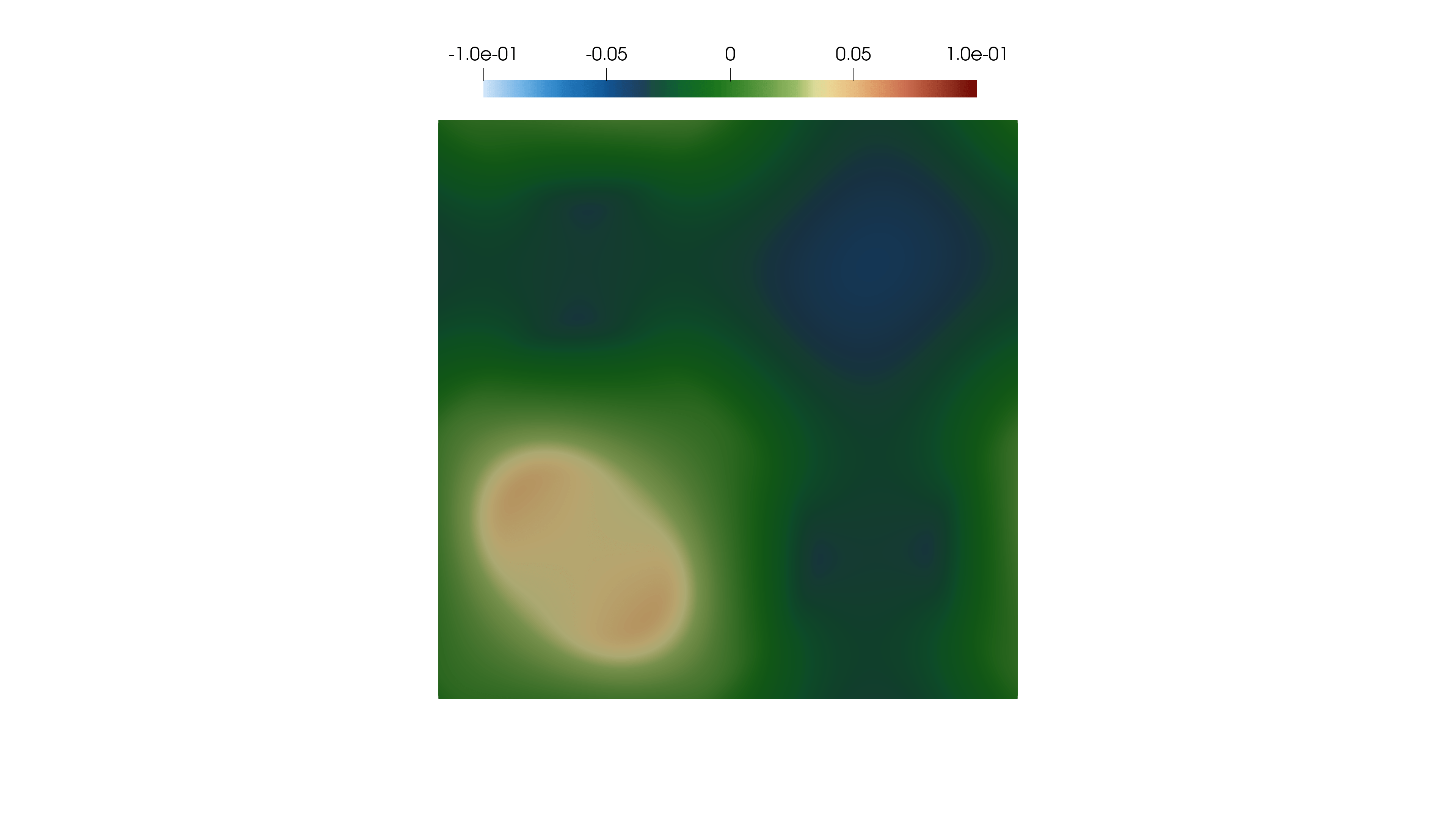} \\
                $t=0.1$ & $t=0.5$ & $t=1$ & $t=2$ & $t=5$  \\
            \end{tabular}
            \caption{Snapshots of the temporal evolution of $\theta_h-1$ for $i=A$ (top) and $i=B$ (bottom). The colorbar is cropped after $\pm0.1$ although at the initial stage these range is strongly exceeded.}
            \label{fig:ex_melt_theta}
        \end{figure}
        \begin{figure}[htbp!]
            \centering
         \begin{tabular}{cc@{}}
                %\hspace{-0.05\linewidth}
                \includegraphics[width=0.4\linewidth,trim={1.5cm 0cm 0cm 0cm},clip]{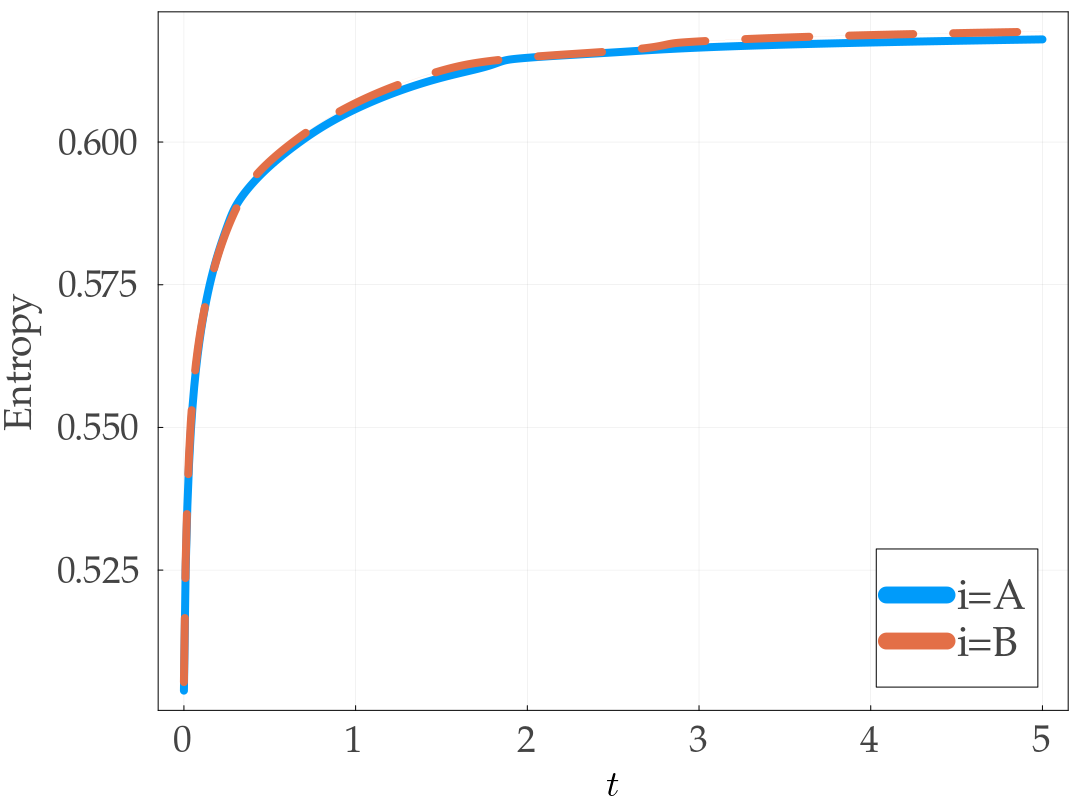}
                &
                \includegraphics[width=0.4\linewidth,trim={1.5cm 0cm 0cm 0cm},clip]{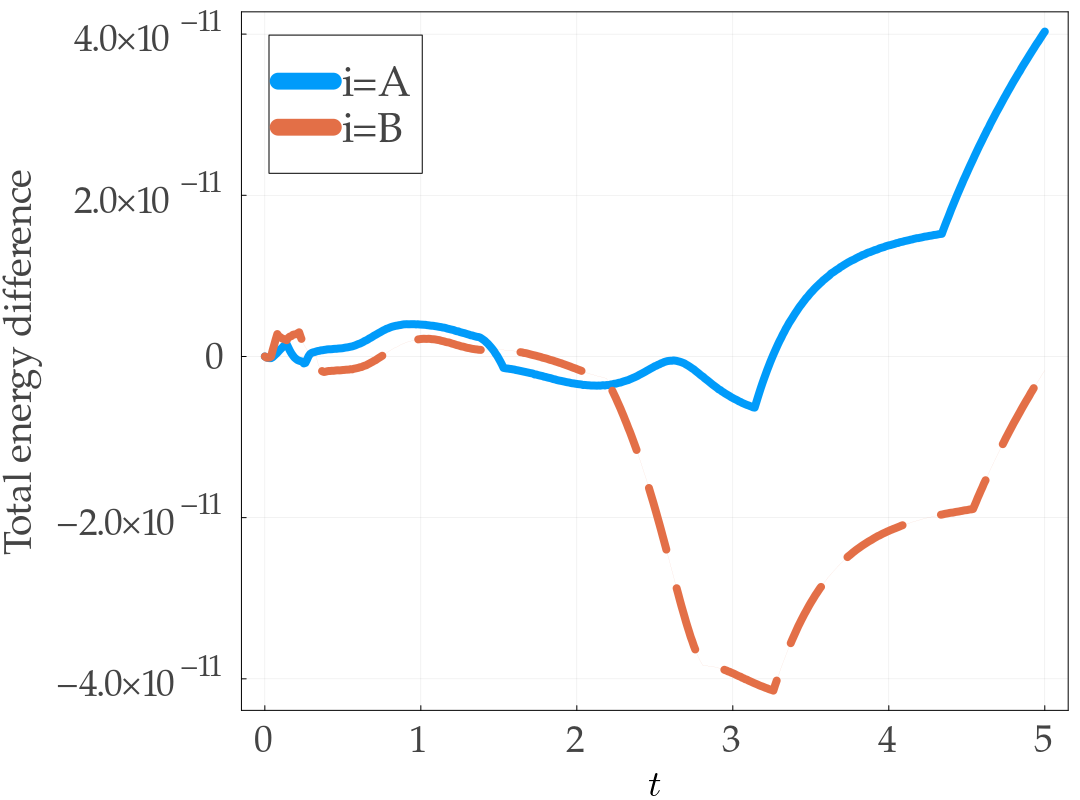}
            \end{tabular}
            \caption{Temporal evolution of the integrated entropy density $s_h$ (left) and conservation error of the integrated total energy density $e_{\mathrm{tot},h}$ (right)}
            \label{fig:ex_melt_struct}
        \end{figure}

\section{Conclusion \& Outlook}
In this work, we developed a numerical method based on the entropy formulation that conserves total energy and guarantees exact entropy production. The main idea was to reformulate the system using entropy as the main variable, rather than the previous approach, which treated entropy as a non-linear function of the phase-field and temperature. The properties of the scheme and its experimental order of convergence are demonstrated and align with previous results and expectations. In the future, we plan to consider the error analysis for such a discretisation, which is formally second-order in time and first/second-order in space. Furthermore, we aim to extend such discretisations for applications in powder bed fusion process, i.e. for a model similar to \cite{yangscripta2020}.

\bibliographystyle{pamm}
\bibliography{S18_Brunk_xx_v1}
\end{document}